\documentclass[10pt,a4paper]{amsart}

\usepackage{verbatim}
\usepackage{geometry}
\usepackage{enumerate}
\usepackage{amsmath,amsfonts,amssymb}
\usepackage{esint}
\usepackage[colorlinks=true, linkcolor=blue, citecolor=blue, urlcolor=blue]{hyperref}
\usepackage{aligned-overset}
\usepackage{xcolor}
\usepackage{upgreek}

\newgeometry{tmargin=2.4cm, bmargin=3cm, lmargin=2.2cm, rmargin=2.2cm}

\numberwithin{equation}{section}

\DeclareMathOperator*{\id}{Id}

\title[Quasilinear Elliptic Equations with Measure Data and Matrix-Valued Weight]{Regularity Estimates for Quasilinear Elliptic Equations with Measure Data and Matrix-Valued Weight}
 
\author[Byun]{Sun-Sig Byun}
\address{Department of Mathematical Sciences and Research Institute of Mathematics, 
Seoul National University, Seoul 08826, Republic of Korea}
\email{byun@snu.ac.kr}

\author[Li]{Xiatong Li}
\address{Department of Mathematical Sciences, 
Seoul National University, Seoul 08826, Republic of Korea}
\email{lixiatong@snu.ac.kr}

\author[Youn]{Yeonghun Youn$^{\ast}$}
\address{Department of Mathematics Education,
Incheon National University, Incheon 21999, Republic of Korea}
\email{yeonghunyoun@inu.ac.kr}

\thanks{$^{\ast}$ Corresponding author.}

\date{}
\begin{document}



\maketitle \sloppy

\thispagestyle{empty}

\def\mean#1{\mathchoice%
          {\mathop{\kern 0.2em\vrule width 0.6em height 0.69678ex depth -0.58065ex
                  \kern -0.8em \intop}\nolimits_{\kern -0.4em#1}}%
          {\mathop{\kern 0.1em\vrule width 0.5em height 0.69678ex depth -0.60387ex
                  \kern -0.6em \intop}\nolimits_{#1}}%
          {\mathop{\kern 0.1em\vrule width 0.5em height 0.69678ex
              depth -0.60387ex
                  \kern -0.6em \intop}\nolimits_{#1}}%
          {\mathop{\kern 0.1em\vrule width 0.5em height 0.69678ex depth -0.60387ex
                  \kern -0.6em \intop}\nolimits_{#1}}}
          
\newcommand{\dv}{{\rm div}}
\newcommand{\A}[1]{{\mathbb{A}(#1)}}
\newcommand{\AB}[2]{{\mathbb{A}_{#1}(#2)}}
\def\R{{\mathbb{R}}}
\def\N{{\mathbb{N}}}
\def\dx{{\,dx}}
\def\dy{{\,dy}}
\def\ds{{\,ds}}

\newcommand{\logomean}[1]{\langle \omega \rangle^{\log}_{#1}}
\newcommand{\logMmean}[1]{\langle \mathbb{M} \rangle^{\log}_{#1}}
\newcommand{\bmo}{\mathrm{BMO}}

\def\rn{{\mathbb{R}^{n}}}
\def\rnn{{\mathbb{R}^{n\times n}}}

\newtheorem{coro}{\bf Corollary}[section]
\newtheorem{theo}[coro]{\bf Theorem} 
\newtheorem{lem}[coro]{\bf Lemma}
\newtheorem{rem}[coro]{\bf Remark} 
\newtheorem{defi}[coro]{\bf Definition} 
\newtheorem{ex}[coro]{\bf Example} 
\newtheorem{fact}[coro]{\bf Fact} 
\newtheorem{prop}[coro]{\bf Proposition}

\begin{abstract}
This paper is concerned with the existence and local gradient regularity of solutions to a quasilinear elliptic Dirichlet problem with finite Radon measure and a degenerate matrix-valued weight. Under a {small log-BMO condition on $\mathbb M$}, we first prove the existence of a SOLA (solution obtained by a limit of approximations) and then establish a local Calder\'on-Zygmund estimate for $|\mathbb M|^{\frac{1}{p-1}}\mathbb M Du$ in terms of the fractional maximal function $\mathcal M_1(\mu)^{\frac{1}{p-1}}$.
\end{abstract}

\section{Introduction}
We study a quasilinear elliptic equation with a finite Radon measure and a degenerate matrix-valued weight. More precisely, we consider the following Dirichlet problem:
\begin{equation}\label{main.eq}
  \left\{ \begin{aligned}
      -\dv \big(|\mathbb{M}(x)Du|^{p-2} \mathbb{M}^2(x) Du\big) &= \mu  &&  in  \; \Omega,\\
                      u &= 0 && on  \; \partial\Omega,
    \end{aligned} \right.
\end{equation} 
where $\Omega \subset \mathbb{R}^n$ is a bounded domain, $n \ge2$, {and} $\mu \in \mathcal{M}(\Omega)$ is a finite Radon measure. 
We extend $\mu$ by zero outside $\Omega$, so that $|\mu|(\Omega)=|\mu|(\mathbb{R}^n)<\infty.$
The parameter $p$ is a constant satisfying 
\begin{equation}\label{p}
2-\frac{1}{n}< p\le n.
\end{equation}
Throughout the paper, $\mathbb{M}: \rn \to \rnn$ is a symmetric and positive definite matrix-valued weight with a uniformly bounded condition number, i.e., for some constant $\Lambda \ge 1$,
\begin{equation}\label{Balance Condition}
  |\mathbb{M}(x)|\,|\mathbb{M}^{-1}(x)| \le \Lambda \quad \text{for \; a.e.} \; x \in \rn,
\end{equation}
where $|\cdot|$ is the spectral norm. We also define the associated scalar weight 
\begin{equation*}\label{norm}
\omega(x) = |\mathbb{M}(x)|= \sup\limits_{\substack{\xi \in \mathbb{R}^n, |\xi|\le 1}} |\mathbb{M}(x)\xi|.
\end{equation*}
Then \eqref{Balance Condition} is equivalent to 
\begin{equation}\label{Balance-1}
\Lambda^{-1}\,\omega(x)\,|\xi|
\le
\bigl|\mathbb{M}(x)\,\xi\bigr|
\le
\omega(x)\,|\xi|
\quad \text{for any }\xi\in\mathbb{R}^n \text{ and for a.e. } x\in\mathbb{R}^n .
\end{equation}

In the case $\mathbb{M}$ is the identity matrix, \eqref{main.eq} reduces to the classical $p$-Laplace equation with bounded Radon measure
\begin{equation}\label{p.without.m}
 -\dv\bigl(|Du|^{p-2}Du\bigr)=\mu
 \quad \text{in } \Omega.
\end{equation}
When $p\le n$, weak solution is unavailable in general, since the measure $\mu$ may fail to belong to the dual space of $W^{1,p}_{0}(\Omega)$. This leads to generalized notions of solution, including SOLAs, entropy solutions, and renormalized solutions.
In particular, the existence theory for \eqref{p.without.m} with $L^1$ or measure data  is well established, see \cite{BG1989,BG1992} for SOLA, \cite{BBGGPV1995,BGO1996} for entropy solutions, and \cite{DMOP1999} for renormalized solutions.
Meanwhile, when $p>n$, by Morrey's inequality together with a duality argument, the measure data $\mu$ defines a bounded linear functional on $W^{1,p}_{0}(\Omega)$. 
Hence, the existence and uniqueness of weak solution to \eqref{p.without.m} with Dirichlet boundary condition follow from the standard monotone operator theory given in \cite{Showalter1997}. 
Throughout this paper, we consider solutions in the sense of SOLA as given in Definition~\ref{defi:SOLA} under the assumption \eqref{p}.
The explicit form of the fundamental solution to the $p$-Laplace equation naturally leads to the assumption $p>2-\frac{1}{n},$ which ensures that the gradient of a SOLA is integrable.
 It is worth emphasizing that the uniqueness of SOLA remains open in general, see \cite{AP1995,JBSJ1964} for counterexamples. 

Another important direction is the regularity theory for measure data problems.
The pointwise estimates for SOLA in terms of Wolff potentials were obtained in the pioneering works by Kilpel\"ainen and Mal\'y \cite{KM1992,KM1994}.
Later, gradient potential estimates were established in \cite{DM2010,KM2013,GM2011}. 
For the summary of the potential estimate, we refer to \cite{KM2014}.
Mingione established fractional differentiability and integrability results in \cite{GM2007}, and obtained local gradient estimates, including Morrey and Lorentz space bounds, for data below the natural duality exponent in \cite{GM2010}.

Calder\'on-Zygmund estimates for equations with measure data are formulated in terms of the 1-fractional maximal function
\begin{equation*}\label{def.M1}
\mathcal{M}_1(\mu)(x)
:= \sup_{r>0}
\frac{r\,|\mu|\bigl(B_r(x)\bigr)}{|B_r(x)|}
\qquad x\in\mathbb{R}^n.
\end{equation*}
For quasilinear equations of $p$-Laplacian type, Phuc \cite{Phuc2014} established global nonlinear Muckenhoupt-Wheeden type estimates in terms of $\mathcal M_1(\mu)$. 
It is proved in \cite[Theorem 1.4]{Phuc2014} that renormalized solutions to \eqref{p.without.m} satisfy the implication
$$
 \mathcal M_1(\mu)^{\frac1{p-1}}
 \in L^{q}(\Omega)
 \quad\Longrightarrow\quad
 Du\in L^{q}(\Omega) \qquad \text{for} \; 0<q<\infty.
$$
This theory was extended to the singular range ${\frac{3n-2}{2n-1}}<p\le 2-{\frac1n}$ in \cite{NP2019,Tran2019,TN2020}, and to nonstandard growth measure data problems in \cite{BCY2021,BOP2017,SY2026,SYZ2026}.
Motivated by \cite[Theorem 1.4]{Phuc2014}, the present paper extends the Calder\'on-Zygmund estimate to measure data problems involving a matrix-valued weight.

In the classical uniformly elliptic setting, the coefficient matrix $\mathbb {B}$ is assumed to satisfy
\begin{equation*}\label{coe.cond.1}
\lambda |\xi|^2
\le
\left\langle \mathbb {B}(x)\xi,\xi\right\rangle
\le
\Lambda |\xi|^2.
\end{equation*}
For linear equations of divergence form with such coefficients, higher integrability results were established in \cite{NGM1963} when $\mathbb {B}$ is a bounded measurable matrix, while global Calder\'on-Zygmund estimates for $\mathrm{VMO}$ coefficients were obtained in \cite{GDF1996}.
The latter were extended to nonlinear equations of $p$-Laplacian type in \cite{KZ1999,KZ2001}.
These results rely essentially on uniform ellipticity of coefficients and therefore exclude degenerate weights, such as $|x|^{\pm\varepsilon}\id$.

A classical regularity theory for degenerate elliptic equations was developed by Fabes, Kenig, and Serapioni \cite{FKS1982}, who considered the matrix-valued weight $\mathbb{B}$ satisfying
\begin{equation*}\label{coe.cond.2}
 \Lambda^{-1} \,\upomega(x)|\xi|^2
 \leq
 \langle \mathbb {B}(x)\xi,\xi\rangle
 \leq
 \Lambda \, \upomega(x)|\xi|^2,
\end{equation*}
where $\upomega$ is in the Muckenhoupt class $\mathcal{A}_2$, and established local H\"older regularity results. In the linear case, local and global weighted gradient estimates were proved in \cite{CMP2018} in terms of the weighted BMO norm of $\mathbb {B}$.
More recently, the authors of \cite{BDGPN2022} introduced a small log-BMO condition on the matrix-valued weight and established local Calder\'on-Zygmund estimates for both linear and nonlinear equations of divergence form.
These local estimates were extended to global ones in \cite{BBDL2023}, and further generalized in \cite{BY2025} to equations with variable exponent growth.
It is worth noting that the main results in \cite{BBDL2023,BDGPN2022,BY2025} differ from {those in} \cite{CMP2018} since they treat the weight {as a multiplier rather than as a measure.} 

In this paper, we impose a small log-BMO condition on the matrix-valued weight in the spirit of \cite{BDGPN2022}.
More precisely, for any given constant $R_0>0$, we define
\begin{equation}\label{small.log.bmo}
{[\log\mathbb{M}]_{\operatorname{BMO}_{R_0}(\mathbb{R}^n)}}
:=\sup_{y\in\mathbb{R}^n}\sup_{0<r\le R_0}\fint_{B_r(y)}\left|\log \mathbb{M}-\langle \log \mathbb{M}\rangle_{B_r(y)}
\right|\,dx,
\end{equation}
where 
$$\langle\log \mathbb{M}\rangle_{B_r(y)}=\fint_{B_r(y)}\log \mathbb{M}(x)\,dx=\frac{1}{|B_r(y)|}\int_{B_r(y)}\log \mathbb{M}(x)\,dx.$$

{Under some smallness condition on $[\log\mathbb{M}]_{\operatorname{BMO}_{R_0}(\mathbb{R}^n)}$}, the known results are limited to equations with divergence form, for which solutions are understood as weak solutions in the corresponding weighted energy space, see \cite{BDGPN2022}. 
In this paper, we allow the right-hand side to be a finite Radon measure.

As noted after \eqref{p.without.m}, the existence of weak solutions to \eqref{main.eq} is not guaranteed when $p\le n$. 
Indeed, when $p>n$ and $\omega$ is an $\mathcal A_s$-Muckenhoupt weight for any $1<s<\frac pn$, it follows from \cite[Theorem 1.14]{HKM1993} that functions in $W^{1,p}_{0,\omega}(\Omega)$ admit continuous representatives, which allows us to treat the case $p>n$ using the notion of weak solutions.
Throughout this paper, we assume that $2-\frac{1}{n}<p\le n$. We introduce the following definition of SOLA:
\begin{defi}\label{defi:SOLA}
    We say that $u\in W^{1,max\{1, p-1\}}_{0,\omega^{p'}}(\Omega)$ is a SOLA to the problem \eqref{main.eq} if  
\begin{equation*}\label{weak-0}
\int_{\Omega} |\mathbb{M} Du|^{p-2} \mathbb{M}^{2}Du\cdot D\varphi\,dx = \int_{\Omega} \varphi\, d\mu \quad\text{for all } \; \varphi\in C^{\infty}_{c}(\Omega),
\end{equation*}
and there exists a sequence of weak solutions $\{u_k\}_{k\ge1}\subset W^{1,p}_{0,\omega}(\Omega)$ to
\begin{equation}\label{weak-1}
  \left\{ \begin{aligned}
      -\dv \big(|\mathbb{M}(x)Du_k|^{p-2} \mathbb{M}(x)^2 Du_k\big) &= \mu_k  && in\;   \Omega,\\
                      u_k &= 0 && on\;   \partial\Omega,
    \end{aligned} \right.
\end{equation}
such that 
\begin{equation}\label{weak-2}
u_k \to u
\quad \text{in } W^{1,\max\{1,p-1\}}_{0,\omega^{p'}}(\Omega)
\quad \text{as } k\to\infty ,
\end{equation}
where the sequence $\{\mu_k\}_{k\ge1}\subset L^{\infty}(\Omega)$ converges weakly to $\mu$
in the sense of measure and satisfies
\begin{equation}\label{weak-3}
 \limsup_{k\to\infty} |\mu_k|(V)\le |\mu|(\overline V)
\end{equation}
for each open set $V \Subset \Omega$.
\end{defi}

We next establish the existence of SOLA to \eqref{main.eq}. More precisely, under a sufficiently small log-BMO condition on $\mathbb{M}$, we have the following result.
\begin{theo}\label{Existence.of.SOLA}
Assume that \eqref{p} holds. There exists a constant $\kappa_0=\kappa_0(n,p,\Lambda)>0$ such that if $[\log\mathbb{M}]_{\operatorname{BMO}_{R_0}(\mathbb{R}^n)}\le \kappa_0$ for some $R_0>0$, there is a SOLA to \eqref{main.eq}.
\end{theo}

Our second main result is a local Calder\'on-Zygmund estimate for a weighted gradient of SOLA. Recall that in \cite{BDGPN2022}, the authors considered the following elliptic equations with matrix-valued weights:
\begin{equation}\label{eq.D}
 -\dv \big(|\mathbb{M}(x)Du|^{p-2} \mathbb{M}^2(x) Du\big) = -\dv \big(|\mathbb{M}(x)G|^{p-2} \mathbb{M}^2(x) G\big).
 \end{equation}
Under a suitable small log-BMO condition on $\mathbb{M}$, the authors established the following transfer of regularity from $G : \Omega \rightarrow \mathbb{R}^n$ to the gradient of weak solution $u$:
$$
 \mathbb{M} G \in L^{q}_{\mathrm{loc}}(\Omega)
 \quad\Longrightarrow\quad
\mathbb{M} Du\in L^{q}_{\mathrm{loc}}(\Omega)
$$
for any $q \in [p,\infty).$
The equation \eqref{main.eq} has the same left-hand side as \eqref{eq.D}, but its right-hand side is a finite Radon measure, whereas the right-hand side of \eqref{eq.D} is in divergence form. Therefore, the regularity estimates obtained there cannot be directly applied to the present problem.
For equations with measure data, the fractional maximal function $\mathcal M_1(\mu)^{\frac{1}{p-1}}$ naturally appears in gradient estimates. This motivates us to establish a local Calder\'on-Zygmund estimate for $|\mathbb M|^{\frac{1}{p-1}}\mathbb M Du$ in terms of $\mathcal M_1(\mu)^{\frac{1}{p-1}}$.

\begin{theo}\label{theo:main}
    Assume that \eqref{p} holds and let $1<q<\infty$. There exists a constant ${\kappa}={\kappa}(n, p,\Lambda)>0$ such that if
$$[\log\mathbb{M}]_{\operatorname{BMO}_{R_0}(\mathbb{R}^n)}<\frac{{\kappa}}{q}$$
for some $R_0>0$, then for any SOLA $u$ to \eqref{main.eq}, we have 
\begin{equation}\label{main.thm.est}
\fint_{B_R} (|\mathbb{M}|^{\frac{1}{p-1}}|\mathbb{M}Du| )^q \,dx
\le
c\,\left(\fint_{B_{2R}}|\mathbb{M}|^{\frac{1}{p-1}}|\mathbb{M}Du|\,dx\right)^q
+
c\fint_{B_{2R}} (\mathcal{M}_1(\mu) )^{\frac{q}{p-1}}\,dx
\end{equation}
for every ball $B_{2R}\Subset\Omega$ and $2R\le R_0$, where $c=c(n,p,q,\Lambda)>0$.
\end{theo}
The conclusion of Theorem~\ref{theo:main} implies that 
$$
 \mathcal M_{1}(\mu)^{\frac1{p-1}}
 \in L^{q}_{\mathrm{loc}}(\Omega)
 \quad\Longrightarrow\quad
|\mathbb{M}|^{\frac{1}{p-1}}\mathbb{M}Du \in L^{q}_{\mathrm{loc}}(\Omega).
$$
When $\mathbb M=\mathrm{Id}$, the weighted gradient reduces to the standard gradient, namely,
$$|\mathbb M|^{\frac{1}{p-1}}\mathbb M Du=Du.$$
Therefore, estimate \eqref{main.thm.est} recovers the classical Calder\'on-Zygmund estimates for elliptic equations with measure data.

The remainder of the paper is organized as follows. Section~2 introduces the notation and collects the preliminary results used throughout the paper. In Section~3, we prove the existence result stated in Theorem~\ref{Existence.of.SOLA}. In Section~4, we derive the reverse H\"older inequalities and the comparison estimates that form the basis of the subsequent regularity analysis. Finally, in Section~5, we prove Theorem~\ref{theo:main} by combining these comparison estimates with a covering argument applied to the super-level sets of $|\mathbb M|^{\frac{1}{p-1}}\mathbb M Du$.

\section{Preliminaries}

In this section, we introduce basic notations and a list of preliminary results which will be used later. For $r>0$ and $y\in\mathbb{R}^n$, we denote by $B_r(y):=\{x\in\mathbb{R}^n:\ |x-y|<r\}$ and $\partial B_r(y):=\{x\in\mathbb{R}^n:\ |x-y|=r\}.$ When the center is clear from the context, we simply write $B_r=B_r(y)$ and $\partial B_r=\partial B_r(y)$.
In what follows, $c$ denotes a positive constant depending only on $n$, $p$, $q$, and $\Lambda$, which may vary from line to line. Dependence on any additional parameters will be indicated explicitly. For any $1<p<\infty$, $p'$ denotes the conjugate exponent of $p$.

To recall the basic properties of matrix-valued weights and related function spaces, we employ $\mathbb{M}: \rn \to \rnn$ and $\omega: \rn \to \R$ to represent general matrix and scalar weights, respectively, throughout this section.
For some domain $U \subset \mathbb{R}^n$, $U \neq \emptyset$, and $1<p<\infty$, a weight $\omega \in L^1_{\mathrm{loc}}(\mathbb{R}^n)$ is called an $\mathcal{A}_p(U)$-Muckenhoupt weight if
\begin{equation}\label{Ap}
[\omega]_{\mathcal{A}_p(U)}
:= \sup_{B \subset U}\left( \fint_B \omega \, dx \right)\left( \fint_B \omega^{-\frac{1}{p-1}} \, dx \right)^{p-1}< \infty.
\end{equation}
When $\omega^p \in \mathcal{A}_{p}(U)$, we define the weighted Lebesgue space $L^p_{\omega}(U)$ by
$$L^p_{\omega}(U)
:=\bigl\{\,f: U\to\mathbb{R}:\ \omega f\in L^p(U)\bigr\},
$$
equipped with the norm
$$
\|f\|_{L^{p}_{\omega}(U)}
=
\|\omega f\|_{L^{p}(U)}.
$$
It is well known that $L^p_{\omega}(U)$ is a Banach space. Let $W^{1,p}(U)$ be the classical Sobolev space, $W^{1,p}_{\mathrm{loc}}(U)$ be its local version and $W^{1,p}_0(U)$ be the closure of $C_c^\infty(U)$ in $W^{1,p}(U)$.
We define the weighted Sobolev space $W^{1,p}_{\omega}(U)$ as the collection of all
$u\in W^{1,1}(U)$ such that
$$
u\in L^p_{\omega}(U)
\quad \text{and} \quad
Du \in L^p_{\omega}(U).
$$
We endow $W^{1,p}_{\omega}(U)$ with the norm
$$
\|u\|_{W^{1,p}_{\omega}(U)}
:= \|u\|_{L^p_{\omega}(U)}+\|Du\|_{L^p_{\omega}(U)}.
$$
Finally, $W^{1,p}_{0,\omega}(U)$ denotes the closure of $C_c^\infty(U)$ in $W^{1,p}_{\omega}(U)$.

We now define the logarithmic means as
$$
\logomean{U}
:= \exp\bigl(\langle \log \omega\rangle_{U}\bigr),
\qquad
\logMmean{U}
:= \exp\bigl(\langle \log \mathbb{M}\rangle_{U}\bigr).
$$
The logarithmic mean is compatible with inversion in the following sense:
$$
\left\langle \frac{1}{\omega}\right\rangle^{\log}_{U}
= \exp\bigl(-\langle \log \omega\rangle_{U}\bigr)
= \frac{1}{\langle \omega\rangle^{\log}_{U}}.
$$
Using the relations $\log(\mathbb{M}^{-1})=-\log\mathbb{M}$ and
$(\exp(\mathbb{M}))^{-1}=\exp(-\mathbb{M})$, we also obtain
$$
\langle \mathbb{M}^{-1}\rangle^{\log}_{U}
= \exp\bigl(-\langle \log \mathbb{M}\rangle_{U}\bigr)
= \bigl(\langle \mathbb{M}\rangle^{\log}_{U}\bigr)^{-1}.
$$
The condition \eqref{Balance-1} yields
$$
\log\bigl(\Lambda^{-1}\omega(x)\bigr) \id
\leq \log \mathbb{M}(x)
\leq (\log \omega(x)) \id.
$$

We state the following weighted Poincar\'e estimates as in \cite[Proposition~3]{BDGPN2022}.
\begin{lem}\label{poincare-0}
Let $1<p<\infty$ and $\theta\in(0,1]$ be such that $\theta p\ge\max\left\{1,\frac{np}{n+p}\right\}.$
Let $B_r$ be a ball and assume that $\omega$ is a weight on $B_{2r}$ with
\begin{equation*}\label{poincare.asmp}
\sup_{B_{\rho}(x_0)\subset B_{2r}}
\left( \fint_{B_{\rho}(x_0)} \omega^p\,dx \right)^{\frac1p}
\left( \fint_{B_{\rho}(x_0)} \omega^{-(\theta p)'}\,dx \right)^{\frac1{(\theta p)'}}
\le c_1
\end{equation*}
for some constant $c_1$.
Then for any $f \in W^{1,p}_{\omega}(B_r)$, there holds
$$
\left(
\fint_{B_r}
\left|\frac{f - \langle f \rangle_{B_r}}{r}\right|^p
\omega^p\,dx
\right)^{\frac1p}
\le
c
\left(
\fint_{B_r} (|Df|\omega)^{\theta p}\,dx
\right)^{\frac1{\theta p}},
$$
where $c=c(c_1,n,p)$ is some constant. 
\end{lem}

Lemma~\ref{poincare-0} is local in nature and will be applied in Section~4 and Section~5. For the existence theory developed in Section~3, however, we need the estimate on the entire domain $\Omega$. We therefore derive the global version of Lemma~\ref{poincare-0}.
Since $d:=\operatorname{diam}(\Omega) \in (0, \infty)$, we have $\Omega\subset B_{2d}(x_0)$ for any $x_0\in \Omega$.
Let $f \in W^{1,p}_{\omega}(\Omega)$ and take the zero extension of $f$ on $B_{2d}\setminus\Omega$.
Then, following the proof of \cite[Proposition~2.2]{BBDL2023},  we obtain the global Poincar\'e inequality as follows:
\begin{lem}\label{poincare.global}
Suppose that the parameters $p$ and $\theta$ satisfy the same assumptions as in Lemma~\ref{poincare-0}.
Assume that $\omega$ is a weight on $B_{2d}$ satisfying
\begin{equation*}\label{poincare.asmp.global}
\sup_{B_{\rho}(x_0)\subset B_{2d}}
\left( \fint_{B_{\rho}(x_0)} \omega^p\,dx \right)^{\frac1p}
\left( \fint_{B_{\rho}(x_0)} \omega^{-(\theta p)'}\,dx \right)^{\frac1{(\theta p)'}}
\le c_2
\end{equation*}
for some constant $c_2>0$.
Then for any $f \in W^{1,p}_{0,\omega}(\Omega)$, there holds
$$
\left(\fint_{\Omega} |f|^p \omega^p \, dx\right)^{\frac{1}{p}}
\leq
c d \left(\fint_{\Omega} (|Df|\omega)^{\theta p} \, dx\right)^{\frac{1}{\theta p}},$$
where $c=c(c_2,n,p)>0$.
\end{lem}

Now, we define the local $\mathrm{BMO}$ space on a ball $B_r$ as the set of functions
$f\in L^1(B_r)$ such that the seminorm
\begin{equation}\label{BMO space}
{[f]_{\mathrm{BMO}(B_r)}}
:= \sup_{B_{\rho}(x) \subset B_{r}}
\left(
\frac{1}{|B_\rho(x)|}\int_{B_\rho(x)}
\bigl|f(y)-\langle f\rangle_{B_{\rho}(x)}\bigr|\,dy
\right)
<\infty.
\end{equation}
Combining \eqref{BMO space} and recalling the definition of $\mathcal{A}_p$-Muckenhoupt weight in \eqref{Ap}, we recall a sequence of lemmas for the properties of matrix-valued weights and scalar weights in \cite[Section 2]{BDGPN2022}.
For simplicity of notations, from Lemma~\ref{lem.oM}~to~\ref{lem:logBMO}, we write $B$ as any ball in $\R^n$.

\begin{lem}\label{lem.oM}
For a matrix-valued weight $\mathbb{M}$ and $\omega=|\mathbb{M}|$, there holds
$$
\fint_B\left|\log \omega(x)-\langle \log \omega\rangle_B\right|\,dx
\leq
2\fint_B
\left|\log \mathbb{M}(x)-\langle \log \mathbb{M}\rangle_B\right|\,dx.
$$
Moreover,
$$
{[\log \omega]_{\operatorname{BMO}(B)}}
 \le 2{[\log\mathbb{M}]_{\operatorname{BMO}(B)}}.
$$
\end{lem}

As a consequence of John-Nirenberg Lemma, we have the following:
\begin{lem}\label{lem.JNL}
There exist constants $\gamma=\gamma(n,\Lambda)>0$ and $c=c(n,\Lambda)>0$ such that the following holds. If $s\ge 1$ and $\mathbb{M}$ is a matrix-valued weight with ${[\log\mathbb{M}]_{\operatorname{BMO}(B)}} \leq \frac{\gamma}{s},$
then
$$
\left(\fint_B\left(\frac{\bigl|\mathbb{M}-\langle \mathbb{M}\rangle_B^{\log}\bigr|}{|\langle \mathbb{M}\rangle_B^{\log}|}\right)^s\,dx\right)^{\frac{1}{s}}
\leq
c s {[\log\mathbb{M}]_{\operatorname{BMO}(B)}}.
$$
The same holds with $\mathbb{M}$ replaced by a scalar weight $\omega$.
\end{lem}

As a consequence of Lemma~\ref{lem.JNL}, the next lemma shows that, {under the small log-BMO condition on the matrix-valued weight}, the weight and its reciprocal satisfy suitable integrability estimates, and the weight belongs to an appropriate Muckenhoupt class.
\begin{lem}\label{lem:logBMO}
There exists a constant $\gamma=\gamma(n,\Lambda)>0$ such that the following hold for all weights $\omega$.
\begin{enumerate}[(a)]
\item If $[\log\omega]_{\mathrm{BMO}(B)}\le \dfrac{\gamma}{s}$ with $s\ge 1$, then
$$
\left(\fint_B \omega^{s}\,dx\right)^{\frac1s}\le 2\,\langle \omega\rangle^{\log}_{B}.
$$

\item If $[\log\omega]_{\mathrm{BMO}(B)}\le \dfrac{\gamma}{s}$ with $s\ge 1$, then
$$
\left(\fint_B \omega^{-s}\,dx\right)^{\frac1s}\le \frac{2}{\langle \omega\rangle^{\log}_{B}}.
$$

\item If $[\log\omega]_{\mathrm{BMO}(B)}\le \gamma\min\!\left\{\frac1p,\frac1{p'}\right\}$
with $1<p<\infty$, then $\omega^p$ is an $\mathcal{A}_p(B)$-Muckenhoupt weight and
$$
[\omega^p]_{\mathcal{A}_p(B)}^{\frac{1}{p}} =
\sup_{B'\subset B}
\left(\fint_{B'}\omega^{p}\,dx\right)^{\frac1p}
\left(\fint_{B'}\omega^{-p'}\,dx\right)^{\frac1{p'}}
\le 4.
$$

\item Let $1<p<\infty$ and $\theta\in(0,1)$ such that $\theta p>1$.
If $[\log\omega]_{\mathrm{BMO}(B)}\le \gamma\min\!\left\{\frac1p,\;1-\frac{1}{\theta p}\right\}$,
then
$$
\sup_{B'\subset B}
\left(\fint_{B'}\omega^{p}\,dx\right)^{\frac1p}
\left(\fint_{B'}\omega^{-(\theta p)'}\,dx\right)^{\frac1{(\theta p)'}}
\le 4.
$$
\end{enumerate}
\end{lem}

It is worth noting that, for every $1<p<\infty$, when $[\log\mathbb{M}]_{\operatorname{BMO}_{R_0}(\mathbb{R}^n)}\le \gamma\min\left\{\frac1p,\frac1{p'}\right\}$, one has $\omega^p \in\mathcal{A}_p(B_r)$ for every ball $B_r \subset \mathbb{R}^n$, even if the given $R_0$ in \eqref{small.log.bmo} is arbitrarily small. Indeed, the case $r\le R_0$ follows directly from Lemma~\ref{lem:logBMO}, whereas the case $r>R_0$ follows by covering $B_r$ with finitely many balls and connecting these balls by finite chains of intersecting balls. For $r>R_0$, the corresponding constants may additionally depend on $r$ and ${R_0}$. This fact is needed in Section~3 to ensure the existence of weak solutions on the whole domain $\Omega$ in the proof of the existence of SOLA.

We also recall the following classical iteration lemma; see \cite[Lemma~4.3]{HL1997}.

\begin{lem}\label{lem-iteration}
Let $\phi$ be a bounded nonnegative function on $[r_1,r_2]$. Suppose that for any
$s_1,s_2$ with $0< r_1 \le s_1 < s_2 \le r_2$,
$$
\phi(s_1)\le \eta \phi(s_2) + \frac{c_1}{(s_2-s_1)^{\alpha}} + c_2,
$$
where $\alpha, c_1, c_2 \ge 0$ and $\eta \in (0,1)$. Then for any
$r_1 \le s_1 < s_2 \le r_2$, it holds that
$$
\phi(r_1)\le c\left[\frac{c_1}{(r_2-r_1)^{\alpha}} + c_2\right]
$$
for some $c=c(\eta,\alpha)>0$.
\end{lem}

We conclude this section by recalling the following standard continuity and monotonicity estimates. 
\begin{lem}
For every $\xi_1,\xi_2 \in \mathbb{R}^n$, there exists a constant $c=c(n,p)>0$ such that
\begin{equation}\label{mono.ineq1}
||\xi_1|^{p-2}\xi_1-|\xi_2|^{p-2}\xi_2|\le c
\begin{cases}
|\xi_1-\xi_2|^{p-1}
& \text{if} \quad 1<p\le 2, \\
\left(|\xi_1|+|\xi_2|\right)^{p-2}
|\xi_1-\xi_2|
& \text{if} \quad p>2,
\end{cases}
\end{equation}
and
\begin{equation}\label{mono.ineq2}
\left\langle |\xi_1|^{p-2}\xi_1-|\xi_2|^{p-2}\xi_2, \xi_1-\xi_2 \right\rangle
\ge c
\begin{cases}
\left(|\xi_1|+|\xi_2|\right)^{p-2}
|\xi_1-\xi_2|^2
& \text{if} \quad 1<p\le2,\\
|\xi_1-\xi_2|^{p}
& \text{if} \quad p> 2,
\end{cases}
\end{equation}
where $\langle \cdot,\cdot \rangle$ is the standard inner
product in $\mathbb{R}^n$.
\end{lem}

\section{Proof of Theorem~\ref{Existence.of.SOLA}}
In this section, we establish the existence of SOLA to problem \eqref{main.eq}.

\begin{proof}[Proof of Theorem~\ref{Existence.of.SOLA}]

The strategy of the proof relies on a standard approximation argument, which is carried out in two steps.
We first derive uniform $L^{\tilde{q}}$-estimate for the weak solutions to \eqref{main.eq} with every $1<\tilde q<\frac{n(p-1)}{n-1}$ under the assumption that $\mu$ is regular.
In the second step, we construct a sequence of solutions to regularized problems; by a compactness argument based on the boundedness from the first step, we then show that a subsequence converges to a limit function, which is shown to be a SOLA to \eqref{main.eq}.

To obtain the uniform estimate on the domain $\Omega$ based on the {small log-BMO condition}, which is local property in nature, in the first step, we will use covering arguments.
Hence, for the simplicity of notation, the generic constant $c$ in this section depends only on $n,p, \tilde{q}, \Lambda, R_0, d$.

\medskip
\noindent\textbf{Step 1: Uniform $L^{\tilde{q}}$-estimate.}
In this step, we assume that $\mu$ belongs to $C^{\infty}(\Omega)\cap L^{\infty}(\Omega)$ to ensure the existence of a unique weak solution $u \in W^{1,p}_{0,\omega}(\Omega)$ to \eqref{main.eq}.
We now assert the following statement:
\begin{quote}
For every $1<\tilde q<\frac{n(p-1)}{n-1}$, there exists {a constant $\kappa_1=\kappa_1(n,p,\tilde q,\Lambda)>0$ such that if $[\log\mathbb{M}]_{\operatorname{BMO}_{R_0}(\mathbb{R}^n)}\le \kappa_1$ for some $R_0>0$}, then there is a constant $c=c(n, p, \tilde q, \Lambda, R_0, d)$ satisfying
\begin{equation}\label{eq.3.1}
\int_{\Omega} \bigl(\omega^{p'}|Du|\bigr)^{\tilde q}\,dx \le c\,|\mu|(\Omega)^{\frac{\tilde q}{p-1}}.
\end{equation}
\end{quote}

Define $u_{\pm}:=\max\{\pm u,0\}$, \(\Omega^{+}:=\{x\in\Omega: u(x)>0\}\), and \(\Omega^{-}:=\Omega\setminus\Omega^{+}\).
Testing \eqref{main.eq} with $\varphi = h^{1-\xi} - (h+u_{\pm})^{1-\xi}\in W^{1,p}_{0,\omega}(\Omega)$ for some constants $h>0$, $\xi>1$ chosen shortly, we have
\begin{equation*}\label{est.u.1}
\int_{\Omega} \frac{\omega^{p}|Du|^{p}}{(h+|u|)^{\xi}}\,dx
\le  c\frac{h^{1-\xi}}{\xi-1}|\mu|(\Omega).
\end{equation*}
Taking
$$
 h = \left(\int_\Omega {\omega}^\frac{\tilde q p'}{p-\tilde q}|u|^{\frac{\xi \tilde q}{p-\tilde q}}\,dx\right)^{\frac{p-\tilde q}{\xi \tilde q}}\left(\int_\Omega {\omega}^\frac{\tilde q p'}{p-\tilde q}\,dx\right)^{-\frac{p-\tilde q}{\tilde q\xi}}, 
 $$
H\"older's inequality gives
\begin{equation}\label{eq3-5}
\begin{aligned}
\int_\Omega ({\omega}^{p'}|Du|)^{\tilde q}\,dx
&\le \left(\int_\Omega \frac{\omega^p|Du|^p}{(h+|u|)^\xi}\,dx\right)^{\frac{\tilde q}{p}}
\left(\int_\Omega {\omega}^\frac{{(p'-1)p\tilde q}}{p-\tilde q}(h+|u|)^{\frac{\xi \tilde q}{p-\tilde q}}\,dx\right)^{\frac{p-\tilde q}{p}} \\
&\le c  \left(\frac{|\mu|(\Omega)}{\xi-1}\right)^{\frac{\tilde q}{p}}h^\frac{\tilde q}{p}
\left( \int_\Omega {\omega}^\frac{\tilde q p'}{p-\tilde q}\,dx\right)^{\frac{p-\tilde q}{p}}.
\end{aligned}
\end{equation}

To establish the estimate \eqref{eq.3.1}, it remains to bound $h$.
Taking $\xi = \frac{n - \tilde{q} + n(p - \tilde{q})}{2(n-\tilde{q})}$, we have
\begin{equation}\label{param.xi}
\frac{n \tilde{q}(\xi -1) p'}{n(p-\tilde{q}) - \xi(n-\tilde{q})}=\frac{n\tilde{q}p'}{n-\tilde{q}} > \max \left\{ \frac{n\tilde{q}}{n-\tilde{q}},\frac{\tilde{q}p'}{p-\tilde{q}} \right\}.
\end{equation}
Assume
\begin{equation}\label{exist.M.cond}
{[\log\mathbb{M}]_{\mathrm{BMO}_{R_0}(\mathbb{R}^n)}} \leq \gamma \max \left\{\tilde{q}', \frac{n\tilde qp'}{n-\tilde q}\right\}^{-1} =: {\kappa_1},
\end{equation}
where $\gamma = \gamma(n, \Lambda) > 0$ is the constant given in Lemma~\ref{lem:logBMO}.
Recalling $2 - \frac{1}{n} < p \leq n$, we note basic inequalities
\begin{equation}\label{param.xi.1}
1 < \tilde{q} < \frac{n(p-1)}{n-1} \leq n, \quad 1 < \xi < \dfrac{n(p-\tilde{q})}{n-\tilde{q}}, \quad \text{and} \quad \frac{\xi \tilde{q}}{p-\tilde{q}} < \frac{n \tilde{q}}{n - \tilde{q}} .
\end{equation}
Now we apply the Poincar\'e type inequality given in Lemma~\ref{poincare.global} to get
\begin{equation}\label{exist.bd.1}
\begin{aligned}
\int_\Omega {\omega}^\frac{\tilde q p'}{p-\tilde q}|u|^{\frac{\xi \tilde q}{p-\tilde q}}\,dx
& = \int_{\Omega} \left( \omega^{p'} |u| \right)^{\frac{\xi \tilde{q}}{p-\tilde{q}}} \omega^{-\frac{\tilde{q} (\xi -1) p'}{p- \tilde{q}}} \dx \\
& \leq c \left( \int_{\Omega} ( \omega^{p'} |u| )^{\frac{n \tilde{q}}{n- \tilde{q}}} \dx \right)^{\frac{(n-\tilde{q})\xi}{n(p-\tilde{q})}} 
\left( \int_{\Omega} \omega^{- \frac{n\tilde{q}p'}{n-\tilde{q}}}\;dx \right)^{1-\frac{(n-\tilde{q})\xi}{n(p-\tilde{q})}} \\
& \leq c \left( \int_{\Omega} ( \omega^{p'} |Du| )^{\tilde{q}} \dx \right)^{\frac{\xi}{p-\tilde{q}}} 
\left( \int_{\Omega} \omega^{- \frac{n\tilde{q}p'}{n-\tilde{q}}}\;dx \right)^{1-\frac{(n-\tilde{q})\xi}{n(p-\tilde{q})}}.
\end{aligned}
\end{equation}
It is worth mentioning that by Lemma~\ref{lem:logBMO} with \eqref{exist.M.cond}, the last line is finite.

Combining \eqref{eq3-5} and \eqref{exist.bd.1}, we have
\begin{equation}\label{eq13}
\begin{aligned}
\int_\Omega ({\omega}^{p'}|Du|)^{\tilde q}\,dx
& \leq c |\mu|(\Omega)^{\frac{\tilde q}{p}}
\left(\int_\Omega (\omega^{p'}|Du| )^{\tilde q} dx\right)^{\frac{1}{p}},
\end{aligned}
\end{equation}
where $c=c(n, p, \tilde q, \Lambda, R_0, d)>0.$ Here we have also used the fact that
\begin{align*}
\left(\int_\Omega \omega^{-\frac{n\tilde{q}p'}{n-\tilde{q}}} dx\right)^{\frac{(n-\tilde{q})\left(n(p-\tilde{q})-n+\tilde{q}\right)}
{np\left(n-\tilde{q}+n(p-\tilde{q})\right)}} 
\left(\int_\Omega \omega^{\frac{\tilde qp'}{p-\tilde q}} dx\right)^{\frac{(p-\tilde{q})\left(n(p-\tilde{q})-n+\tilde{q}\right)}
{p\left(n-\tilde{q}+n(p-\tilde{q})\right)}}
\leq c,
\end{align*}
which is a consequence of covering argument and Lemma~\ref{lem:logBMO} with the assumption \eqref{exist.M.cond}.
Absorbing the last integral term in \eqref{eq13} into the left-hand side, we finally arrive at \eqref{eq.3.1}.

\,

\noindent\textbf{Step 2: Approximation and Convergence.}

By a standard mollification, we construct a sequence $\{\mu_k\}_{k\ge 1} \subset C^\infty(\Omega)\cap L^{\infty}(\Omega)$ such that $\{ \mu_k \}$ converges to $\mu$ weakly in measure.
Then, there exist the corresponding weak solutions $\{u_k\}_{k\ge 1} \subset W^{1,p}_{0,\omega}(\Omega)$ such that each $u_k$ solves \eqref{main.eq} with the datum $\mu_k$.

Now we are going to prove that $\{ u_k \}_{k \ge 1}$ converges to a SOLA to \eqref{main.eq} up to non-relabeled subsequence.
To this end, we take
\begin{equation}\label{kappa.0}
\begin{aligned}
\kappa_0 := \gamma \max \left\{ \tilde{q}' , \frac{n \tilde{q}p' }{n - \tilde{q}} \right\}^{-1} \leq \gamma \max \left\{ 2, \frac{\tilde{q} p'}{\tilde{q} -2 +p} \right\}^{-1}
\end{aligned}
\end{equation}
with
\begin{equation}\label{tilde.q}
\tilde{q} := \frac12\left(\max\{1,p-1\} + \frac{n(p-1)}{n-1}\right) \in \left( \max\{1,p-1\} , \frac{n(p-1)}{n-1} \right),
\end{equation}
and assume
\begin{equation*}
\begin{aligned}
{[\log\mathbb{M}]_{\mathrm{BMO}_{R_0}(\mathbb{R}^n)}< \kappa_0.}
\end{aligned}
\end{equation*}
Then \eqref{eq.3.1} holds true for $u, \mu$ replaced by $u_k, \mu_k$ for each $k \in \mathbb{N}$.
By \eqref{weak-3}, for large enough $k \in \mathbb{N}$, we further have
\begin{equation}\label{ex.step2.bd}
\int_{\Omega} \bigl(\omega^{p'}|Du_k|\bigr)^{\tilde q}\,dx \le c\,|\mu|(\Omega)^{\frac{\tilde q}{p-1}}.
\end{equation}
Moreover, by Lemma~\ref{lem:logBMO} with \eqref{kappa.0}, $\| \omega \|_{L^{\gamma/\kappa_0}(B_{2d})} < \infty$.
We will use this fact later in this section without mentioning.
In the remaining part of this step, we use the temporary notation $c_{*}$ which stands for constants depending on $n, p, \Lambda, R_0, d, |\mu|(\Omega), \| \omega \|_{L^{\gamma/\kappa_0}(B_{2d})} $ and may vary from line to line.

From the fact that $\{u_k\}_{k\ge1}$ is uniformly bounded in $W^{1,\tilde q}_{0,\omega^{p'}}(\Omega)$, which is a reflexive Banach space when $\tilde q>1$, there exists a non-relabeled subsequence satisfying 
$$u_k \rightharpoonup u \quad\text{in } W^{1,\tilde q}_{0,\omega^{p'}}(\Omega).$$
Then the weighted Rellich--Kondrachov compactness theorem in \cite{K1997} implies that, after passing to a further subsequence if necessary, we have
$$
u_k \to u \text{ in } L^{\tilde q}_{\omega^{p'}}(\Omega)
\qquad\text{and}\qquad
u_k \to u \quad \text{a.e. in }\Omega.
$$
Consequently, $u_k$ converges to $u$ in measure with respect to the Lebesgue measure. That is, for any $\varepsilon > 0$ and sufficiently large $k$, we have
$$\mathrm{meas} \left\{x\in\Omega:|u_k-u|>\frac\varepsilon2 \right\}\to 0.$$
 Thus, for $l,m$ sufficiently large,
\begin{equation}\label{meas.est.mn}
\mathrm{meas}\left\{x\in\Omega:|u_l-u_m|>\varepsilon \right\} \le \mathrm{meas}\left\{x\in\Omega:|u_l-u|>\frac\varepsilon2 \right\}+\mathrm{meas}\left\{ x\in\Omega:|u-u_m|>\frac\varepsilon2 \right\} \leq \varepsilon.
\end{equation}

For a fixed $\varepsilon>0$, we consider the truncation function $\psi \in C(\mathbb{R},\mathbb{R})$ defined by
$$
\psi(s)=
\begin{cases}
\varepsilon  & \text{if } s>\varepsilon,\\[4pt]
s  & \text{if } |s|\le \varepsilon,\\[4pt]
-\varepsilon  & \text{if } s<-\varepsilon.
\end{cases}
$$
For any $l, m \in \mathbb{N}$, we test the corresponding equations for both $(\mu_l, u_l)$ and $(\mu_m, u_m)$ with $\varphi=\psi(u_l-u_m) \in W^{1,p}_{0,\omega}(\Omega)$, and then subtract the two equations to obtain
\begin{equation}\label{eq17}
\begin{aligned}
& \int_{D_{l,m,\varepsilon}} \langle|\mathbb{M} D u_l|^{p-2} \mathbb{M} D u_l - |\mathbb{M} D u_m|^{p-2} \mathbb{M} D u_m , \mathbb{M} D u_l - \mathbb{M} D u_m \rangle \,dx \\
& = \int_{\Omega} \psi'(u_l-u_m) \langle |\mathbb{M} D u_l|^{p-2} \mathbb{M} D u_l - |\mathbb{M}D u_m|^{p-2} \mathbb{M} D u_m , \mathbb{M} D u_l - \mathbb{M} D u_m \rangle \,dx \\
& = \int_{\Omega} \psi(u_l-u_m)\,d(\mu_l-\mu_m)
\leq 2 \varepsilon |\mu|(\Omega),
\end{aligned}
\end{equation}
where the set $D_{l,m,\varepsilon}$ is defined by
$$
D_{l,m,\varepsilon}
 := \left\{ x\in\Omega \,\big|\, |u_l(x)-u_m(x)| \le \varepsilon \right\}.
$$
Since the monotonicity estimate \eqref{mono.ineq2} takes different forms in the subquadratic and superquadratic cases, we estimate the left-hand side of \eqref{eq17} separately for $2-\frac{1}{n}<p\leq 2$ and $2<p\leq n.$

\textbf{Case 1:} $2-\frac1n<p\le2.$ 
By {\eqref{mono.ineq2}}, \eqref{Balance-1}, and \eqref{eq17},
\begin{equation*}\label{eq19}
\begin{aligned}
& \int_{D_{l,m,\varepsilon}} \omega^{p}(|Du_l| + |Du_m|)^{p-2}|Du_l-Du_m|^2\,dx \\
& \leq c \int_{D_{l,m,\varepsilon}} (|\mathbb{M}Du_l|+|\mathbb{M}Du_m|)^{p-2}|\mathbb{M}Du_l-\mathbb{M}Du_m|^2\,dx \\
& \leq c \int_{D_{l,m,\varepsilon}}
\langle |\mathbb{M}Du_l|^{p-2}\mathbb{M}Du_l-|\mathbb{M}Du_m|^{p-2}\mathbb{M}Du_m , \mathbb{M}Du_l-\mathbb{M}Du_m \rangle \,dx \\
& \leq c \varepsilon |\mu|(\Omega).
\end{aligned}
\end{equation*}
Hence, by H\"older's inequality,
\begin{equation}\label{eq20}
\begin{aligned}
& \int_{D_{l,m,\varepsilon}} \omega^{p'}|Du_l-Du_m|\,dx \\
&\le
\left(\int_{D_{l,m,\varepsilon}} \omega^{p} (|Du_{l}| + |Du_{m}|)^{p-2}|Du_l-Du_m|^2\,dx\right)^{\frac12} \cdot \\
& \qquad \cdot
\left(\int_{D_{l,m,\varepsilon}} \omega^{p'(3-p)} (|Du_l| + |Du_m|)^{2-p} \,dx\right)^{\frac12} \\
&\le c \varepsilon^{\frac12} |\mu|(\Omega)^{\frac{1}{2}}
\left(\int_{D_{l,m,\varepsilon}} \omega^{p'(3-p)} (|Du_l| + |Du_m|)^{2-p} \,dx\right)^{\frac12}.
\end{aligned}
\end{equation}
When $p=2$, we find
\begin{equation}\label{ex.step2.p2}
\int_{D_{l,m,\varepsilon}} \omega^{p'(3-p)} (|Du_l| + |Du_m|)^{2-p} \,dx \leq \int_{\Omega} \omega^2 \,dx < \infty.
\end{equation}
Moreover, when $2-\frac1n<p<2$, H\"older's inequality with exponents $\frac{\tilde q}{2-p}$ and $\frac{\tilde q}{\tilde q-2+p}$, and \eqref{ex.step2.bd} yield
\begin{equation}\label{ex.step2.ple2}
\begin{aligned}
& \int_{D_{l,m,\varepsilon}} \omega^{p'(3-p)} (|Du_l| + |Du_m|)^{2-p} \,dx \\
&\le
\left(\int_{D_{l,m,\varepsilon}} [\omega^{p' }(|Du_l|+|Du_m|)]^{\tilde q}\,dx\right)^{\frac{2-p}{\tilde q}}
\left(\int_{D_{l,m,\varepsilon}} \omega^{\frac{\tilde qp'}{\tilde q-2+p}}\,dx\right)^{\frac{\tilde q-2+p}{\tilde q}} \\
&\le
c |\mu|(\Omega)^{\frac{2-p}{p-1}} \left(\int_{\Omega} \omega^{\frac{\tilde qp'}{\tilde q-2+p}}\,dx\right)^{\frac{\tilde q-2+p}{\tilde q}}.
\end{aligned}
\end{equation}
Merging \eqref{eq20}, \eqref{ex.step2.p2}, and \eqref{ex.step2.ple2}, we deduce
\begin{equation}\label{eq21}
\begin{aligned}
\int_{D_{l,m,\varepsilon}} \omega^{p'}|Du_l-Du_m|\,dx
\le c_{*} \varepsilon^{\frac12}.
\end{aligned}
\end{equation}

On the other hand, by H\"older's inequality and \eqref{ex.step2.bd}, we have 
\begin{equation}\label{eq22}
\begin{aligned}
&\int_{\Omega\setminus D_{l,m,\varepsilon}} \omega^{p'}|Du_l-Du_m|\,dx\\
&\le
\left(\int_{\Omega\setminus D_{l,m,\varepsilon}} (\omega^{p'}|Du_l-Du_m|)^{\tilde q}\,dx\right)^{\frac{1}{\tilde q}}
\mathrm{meas}\{x\in\Omega:\ |u_l-u_m|>\varepsilon\}^{1-\frac{1}{\tilde q}} \\
&\le c |\mu|(\Omega)^{\frac{1}{p-1}}\,\mathrm{meas}\{x\in\Omega:\ |u_l-u_m|>\varepsilon\}^{1-\frac{1}{\tilde q}}.
\end{aligned}
\end{equation}
It then follows from \eqref{eq22} that
$$
\int_{\Omega\setminus D_{l,m,\varepsilon}} \omega^{p'}|Du_l-Du_m|\,dx
\le c \varepsilon^{1-\frac{1}{\tilde q}} |\mu|(\Omega)^{\frac{1}{p-1}} .
$$
Combining this with \eqref{eq21}, we conclude
$$
\int_\Omega \omega^{p'}|Du_l-Du_m|\,dx
\le c_{*} ( \varepsilon^{\frac12} + \varepsilon^{1-\frac{1}{\tilde q}}),
$$
which means $\{Du_k\}_{k\ge 1}$ is a Cauchy sequence in $L^1_{\omega^{p'}}(\Omega)$, and hence
\begin{equation}\label{eq23}
\begin{aligned}
\int_\Omega \omega^{p'}|Du_k-Du|\,dx\to 0.
\end{aligned}
\end{equation}

Finally, since $0<p-1\le1$, it follows from \eqref{Balance-1}, \eqref{mono.ineq1}, and \eqref{eq23} that
\begin{align*}
|\mathbb{M}Du_k|^{p-2}\mathbb{M}^2Du_k \to |\mathbb{M}Du|^{p-2}\mathbb{M}^2Du \quad \text{in} \quad L^{1}(\Omega)
\end{align*}
and therefore
$$
\int_\Omega |\mathbb{M}Du|^{p-2}\mathbb{M}^2Du\cdot D\varphi\,dx = \int_{\Omega} \varphi \, d \mu ,
\qquad \forall\,\varphi \in C_c^{\infty}(\Omega).
$$

\textbf{Case 2:} $2<p\le n.$ 
By \eqref{Balance-1}, \eqref{mono.ineq2}, and \eqref{eq17}, we have
\begin{equation*}
\begin{aligned}
\int_{D_{l,m,\varepsilon}} \omega^{p}\,|D u_l - D u_m|^{p}\,dx
&\leq c \varepsilon |\mu|(\Omega).
\end{aligned}
\end{equation*}
By H\"older's inequality and Lemma~\ref{lem:logBMO} with \eqref{kappa.0}, it follows that
\begin{equation}\label{eq26}
\begin{aligned}
&\int_{D_{l,m,\varepsilon}} \omega^{p} |D u_l - D u_m|^{p-1}\,dx\\
&\le  \left( \int_{D_{l,m,\varepsilon}} \omega^{p} |D u_l - D u_m|^{p}\,dx
   \right)^{\frac{p-1}{p}}
   \left( \int_{D_{l,m,\varepsilon}} \omega^{p}\,dx \right)^{\frac1p} \\
&\le c \varepsilon^{\frac{p-1}{p}}|\mu|(\Omega)^{\frac{p-1}{p}}
   \left( \int_{\Omega} \omega^{p}\,dx \right)^{\frac1p} 
 \le c_* \varepsilon^{\frac{p-1}{p}}.
\end{aligned}
\end{equation}

On the other hand, combining \eqref{ex.step2.bd} and \eqref{meas.est.mn}, we obtain
\begin{equation}\label{eq28}
\begin{aligned}
&\int_{\Omega\setminus D_{l,m,\varepsilon}} \omega^{p}|Du_l-Du_m|^{p-1}\,dx\\
&\le
\left(\int_{\Omega\setminus D_{l,m,\varepsilon}} \left(\omega^{p'}|Du_l-Du_m|\right)^{\tilde q}dx\right)^{\frac{2(n-1)}{2n-1}}
\mathrm{meas}\{x\in\Omega:\ |u_l-u_m|>\varepsilon\}^{\frac{1}{2n-1}}\\
&\le c\varepsilon^{\frac{1}{2n-1}} |\mu|(\Omega).
\end{aligned}
\end{equation}
Hence, \eqref{eq26} and \eqref{eq28} imply that
$$
\int_\Omega \omega^{p}|Du_l-Du_m|^{p-1}\,dx
\le c_{*} (\varepsilon^{\frac{p-1}{p}} + \varepsilon^{\frac{1}{2n-1}}),
$$
which means that $\{Du_k\}_{k\ge1}$ is a Cauchy sequence in $L^{p-1}_{\omega^{p'}}(\Omega)$.
Therefore,
\begin{equation}\label{eq29}
\begin{aligned}
\int_\Omega \omega^{p}|Du_k-Du|^{p-1}\,dx\to 0.
\end{aligned}
\end{equation}

Finally, from \eqref{Balance-1}, \eqref{mono.ineq1}, \eqref{ex.step2.bd}, and \eqref{eq29}, we obtain
\begin{equation*}
\begin{aligned}
&\int_{\Omega}\bigl||\mathbb{M} D u_k|^{p-2} \mathbb{M}^{2} D u_k-|\mathbb{M} D u|^{p-2} \mathbb{M}^{2} D u\bigr|\,dx\\
&\le c \left(\int_{\Omega}\omega^{p}\bigl(|Du_k|+|Du|\bigr)^{p-1} dx\right)^{\frac{p-2}{p-1}}\left(\int_{\Omega}\omega^{p}|Du_k-Du|^{p-1} dx\right)^{\frac{1}{p-1}} \\
&\le c\left(\int_{\Omega}\omega^{p}|Du_k|^{p-1} dx+\int_{\Omega}\omega^{p}|Du_k-Du|^{p-1} dx\right)^{\frac{p-2}{p-1}}\left(\int_{\Omega}\omega^{p}|Du_k-Du|^{p-1} dx\right)^{\frac{1}{p-1}}
 \to 0.
\end{aligned}
\end{equation*}
At this stage, by following the same reasoning given in the last part of {\bf Case 1}, we conclude that $u$ is a SOLA to \eqref{main.eq} in the case $2 < p \leq n$.
This completes the proof of Theorem~\ref{Existence.of.SOLA}.
\end{proof}

\section{Reverse H\"older inequality and Comparison estimates}

In what follows in this section, we assume that $\mu \in C^\infty(\Omega)\cap L^{\infty}(\Omega)$ and let $u\in W^{1,p}_{0,\omega}(\Omega)$ be the weak solution to \eqref{main.eq}.

Suppose $B_{3r}\subset B_{2R}\subset B_{R_0}$.
 Let $v\in u+W^{1,p}_{0,\omega}(B_{2r})$ be the weak solution to
\begin{equation}\label{eq1}
\left\{
\begin{aligned}
-\operatorname{div}\left(|\mathbb{M}Dv|^{p-2}\mathbb{M}^{2}Dv\right) &= 0 && \text{in } B_{2r},\\
v &= u && \text{on } \partial B_{2r}.
\end{aligned}
\right.
\end{equation}
For simplicity of notation, we introduce the vector field $\mathbb{A} : \R^n \times \R^n \to \R^n$ given by
$$\A{x,\xi} = |\mathbb{M}(x)|^{\frac{1}{p-1}} \mathbb{M}(x) \xi, 
 \quad \forall x \in \R^n, \,\, \xi \in \R^n.$$
Unless there is ambiguity regarding the spatial variable $x$, we shall write $\A{\xi}$ instead of $\A{x, \xi}$.

\subsection{Reverse H\"older inequality}

In this subsection, we establish two reverse H\"older-type estimates for $\A{Dv}$ and $\A{Du}$, respectively.
We begin with the following self-improving property for the weak solution $v$ to homogeneous equation \eqref{eq1}.

\begin{lem}\label{Reverse.H}
Assume that \eqref{p} holds. Let $v\in u+W^{1,p}_{0,\omega}(B_{2r})$ be the weak solution to \eqref{eq1}.
There exists a constant $\kappa_2=\kappa_2(n,p,\Lambda)>0$ such that for every {$\max\{1,\frac{p}{2}\}\le q_0< \infty$} with
$$[\log\mathbb{M}]_{\mathrm{BMO}({B_{2\rho}})}<\frac{\kappa_2}{q_0}$$
there holds
\begin{equation}\label{Reverse}
\left(\fint_{B_\rho} |\A{Dv}|^{q_0}\,dx\right)^{\frac{1}{q_0}}
\le c\fint_{B_{2\rho}} |\A{Dv}|\,dx
\end{equation}
for any $B_{2 \rho} \subset B_{2r}$, where $c=c(n,p,q_0,\Lambda)>0$.
\end{lem}
\begin{proof}
Applying \cite[Theorem~2]{BDGPN2022} to the homogeneous equation \eqref{eq1}, there is a constant $\kappa_3=\kappa_3(n,p,\Lambda)>0$ such that the following higher integrability estimate holds:
\begin{equation}\label{Dening.Thm2}
\left(\fint_{B_\rho} |\mathbb{M}Dv|^{2q_0}\,dx\right)^{\frac{1}{2q_0}}
\le c \left(\fint_{B_{2\rho}} |\mathbb{M}Dv|^{\frac12}\,dx \right)^2,
\end{equation}
provided
$$[\log\mathbb{M}]_{\mathrm{BMO}(B_{2\rho})}\le \frac{\kappa_3}{q_0}.$$
Take 
\begin{equation}\label{gamma2}
\kappa_2=\min\left\{\kappa_3, \frac{(p-1)\gamma}{2}\right\},
\end{equation}
where $\gamma = \gamma(n, \Lambda) > 0$ is the constant given in Lemma~\ref{lem:logBMO}, and assume
\begin{equation*}
[\log\mathbb{M}]_{\mathrm{BMO}({B_{2\rho}})}\le \frac{\kappa_2}{q_0}.
\end{equation*}
It follows from \eqref{Dening.Thm2} and Lemma~\ref{lem:logBMO} that
\begin{equation*}
\begin{aligned}
\fint_{B_\rho} |\A{Dv}|^{q_0}\,dx
&\le \left(\fint_{B_\rho} |\mathbb{M}|^{\frac{2q_0}{p-1}}\,dx\right)^{\frac12}
\left(\fint_{B_\rho} |\mathbb{M}Dv|^{2q_0}\,dx\right)^{\frac12}\\
&\le c\left(\fint_{B_\rho} |\mathbb{M}|^{\frac{2q_0}{p-1}}\,dx\right)^{\frac12}
\left(\fint_{B_{2\rho}} |\mathbb{M}Dv|^{\frac12}\,dx\right)^{2q_0}\\
&\le c\left(\fint_{B_{2\rho}} |\mathbb{M}|^{\frac{2q_0}{p-1}}\,dx\right)^{\frac12}
\left(\fint_{B_{2\rho}} |\mathbb{M}|^{-\frac{1}{p-1}}\,dx\right)^{q_0}
\left(\fint_{B_{2\rho}} |\A{Dv}|\,dx\right)^{q_0}\\
&\le c\left(\fint_{B_{2\rho}} |\A{Dv}|\,dx\right)^{q_0}.
\end{aligned}
\end{equation*}
This completes the proof of Lemma~\ref{Reverse.H}.
\end{proof}

We next establish the following reverse H\"older inequality for $\A{Du}$.
The proof of Lemma~\ref{reverse.Du} relies mainly on Lemma~\ref{Reverse.H} and Lemma~\ref{comparison estimate}.
Although Lemma~\ref{comparison estimate} is stated later for the sake of presentation, its proof is independent of Lemma~\ref{reverse.Du}.

\begin{lem}\label{reverse.Du}
Assume that \eqref{p} holds. Let $u\in W^{1,p}_{0,\omega}(\Omega)$ be the weak solution to \eqref{main.eq}. 
For every $1<\tilde q<\frac{n(p-1)}{n-1}$, there exists a constant ${\kappa_4}={\kappa_4}(n,p,\tilde q,\Lambda)>0$
such that if
$$[\log\mathbb{M}]_{\mathrm{BMO}({B_{3r}})}<{\kappa_4}$$
there holds
\begin{equation}\label{Reverse-G1}
\left(\fint_{B_{2r}} |\A{Du}|^{\tilde q}\,dx\right)^{\frac{1}{\tilde q}}
\le c\fint_{B_{3r}} |\A{Du}|\,dx
+c\left(\frac{|\mu|(B_{3r})}{r^{n-1}}\right)^{\frac{1}{p-1}},
\end{equation}
where $c=c(n,p,{\tilde q},\Lambda)>0$.
\end{lem}
\begin{proof}
Suppose {$2\leq {s}< {t} \leq 3$} and let $v_{t} \in W^{1,p}_{\omega}(B_{{t} r})$ be the weak solution to 
\begin{equation*}
\left\{
\begin{aligned}
-\operatorname{div} \left(|\mathbb{M}Dv_{t}|^{p-2}\mathbb{M}^{2}Dv_{t}\right) &= 0 && \text{in } B_{{t} r},\\
v_{t} &= u && \text{on } \partial B_{{t} r}.
\end{aligned}
\right.
\end{equation*}
Take
\begin{equation}\label{gamma4}
{\kappa_4}:=\min\left\{\kappa_1, \frac{\kappa_2}{p}\right\},
\end{equation}
and assume 
$${[\log\mathbb{M}]_{\mathrm{BMO}({B_{3r}})}}<{\kappa_4}.$$
Applying Lemma~\ref{comparison estimate} to the pair $(u,v_t)$, we have
\begin{equation*}\label{Lem2-1}
\left(\fint_{B_{{t} r}} |\A{Du}-\A{Dv_t}|^{\tilde{q}}\,dx\right)^\frac1{\tilde{q}}
 \le c\left\{\left(\frac{|\mu|(B_{{t} r})}{r^{n-1}}\right)^{\frac{1}{p-1}}+\chi_{\{p<2\}}\left(\frac{|\mu|(B_{{t} r})}{r^{n-1}}\right)\left(\fint_{B_{{t} r}} |\A{Du}|^{\tilde{q}}\,dx\right)^\frac{2-p}{\tilde{q}}\right\}.
\end{equation*}
It follows from \eqref{Reverse} along with covering argument that
\begin{equation}\label{est.lem4.2}
\begin{aligned}
&\left(\fint_{B_{{s} r}} |\A{Du}|^{\tilde q}\,dx\right)^{\frac1{\tilde q}}\\
&\leq c\left(\fint_{B_{{s} r}} |\A{Du}-\A{Dv_t}|^{\tilde q}\,dx\right)^{\frac{1}{\tilde q}} + c\left(\fint_{B_{{s} r}} |\A{Dv_t}|^{\tilde q}\,dx\right)^{\frac1{\tilde q}}\\
&{\leq c\left(\fint_{B_{{s} r}} |\A{Du}-\A{Dv_t}|^{\tilde q}\,dx\right)^{\frac{1}{\tilde q}} + c\left(\fint_{B_{{s} r}} |\A{Dv_t}|^{p}\,dx\right)^{\frac1{p}}}\\
&\leq c\left(\fint_{B_{{s} r}} |\A{Du}-\A{Dv_t}|^{\tilde q}\,dx\right)^{\frac{1}{\tilde q}}+
\frac{c}{({t}-{s})^{n\left(2-\frac1{\tilde q}\right)}}\fint_{B_{{t} r}}|\A{Dv_t}|\,dx\\
&\leq \frac{c}{({t}-{s})^{n\left(2-\frac1{\tilde q}\right)}} \left(\fint_{B_{{t} r}} |\A{Du}-\A{Dv_t}|^{\tilde q}\,dx\right)^{\frac1{\tilde q}}
+\frac{c}{({t}-{s})^{n\left(2-\frac1{\tilde q}\right)}} \fint_{B_{{t} r}} |\A{Du}|\,dx\\
&\leq \frac{c}{({t}-{s})^{n\left(2-\frac1{\tilde q}\right)}}
 \fint_{B_{{t} r}} |\A{Du}|\,dx
+\frac{c}{({t}-{s})^{n\left(2-\frac1{\tilde q}\right)}}
\left(\frac{|\mu|(B_{{t} r})}{r^{n-1}}\right)^{\frac{1}{p-1}}\\
&\quad+ \chi_{\{p<2\}}
\frac{c}{({t}-{s})^{n\left(2-\frac1{\tilde q}\right)}}
\left(\frac{|\mu|(B_{{t} r})}{r^{n-1}}\right)\left(\fint_{B_{{t} r}} |\A{Du}|^{\tilde{q}}\,dx\right)^\frac{2-p}{\tilde{q}}.
\end{aligned}
\end{equation}

If $p\ge 2,$ then \eqref{Reverse-G1} is a direct consequence of \eqref{est.lem4.2}.
We now focus on the case $2-{\frac1n}<p<2$. By applying Young's inequality to \eqref{est.lem4.2} with conjugate exponents
$\frac{1}{2-p}$ and $\frac{1}{p-1}$, we have 
\begin{equation*}\label{Lem2-6}
\begin{aligned}
\left(\fint_{B_{{s} r}} |\A{Du}|^{\tilde q}\,dx\right)^{\frac1{\tilde q}}
&\leq
\frac{1}{2}
\left(\fint_{B_{{t} r}} |\A{Du}|^{\tilde q}\,dx\right)^{\frac{1}{\tilde q}}
+\frac{c}{({t}-{s})^{n\left(2-\frac1{\tilde q}\right)}}
\fint_{B_{{t} r}} |\A{Du}|\,dx\\
&\quad
+\frac{c}{({t}-{s})^{n\left(2-\frac1{\tilde q}\right)}}
\left(\frac{|\mu|(B_{{t} r})}{r^{n-1}}\right)^{\frac{1}{p-1}} 
+\frac{c}
{({t}-{s})^{\frac{n}{p-1}\left(2-\frac1{\tilde q}\right)}}
\left(\frac{|\mu|(B_{{t} r})}{r^{n-1}}\right)^\frac{1}{p-1}.
\end{aligned}
\end{equation*}
Lemma~\ref{lem-iteration} leads to \eqref{Reverse-G1}.
\end{proof}

\begin{rem}\label{rem.comparison}
Lemma~\ref{reverse.Du} is used only to derive the comparison estimate between \eqref{main.eq} and \eqref{eq1} in the case $2-{\frac1n}<p<2$ as stated in Corollary~\ref{comparison.2}. To this end, throughout the remainder of this paper, we fix $\tilde {q}$ as in \eqref{tilde.q}, so that the {constant ${\kappa_4}$ in Lemma~\ref{reverse.Du} depends only on $n$, $p$, and $\Lambda$.}
However, the proof of Lemma~\ref{comparison estimate} still holds for every $1<\tilde{q} < \frac{n(p-1)}{n-1}$.
\end{rem}

\subsection{Comparison estimates between \eqref{main.eq} and \eqref{eq1}}

\begin{lem}\label{comparison estimate}
Assume that \eqref{p} holds. Let $u\in W^{1,p}_{0,\omega}(\Omega)$ be the weak solution to \eqref{main.eq} and $v\in u+W^{1,p}_{0,\omega}(B_{2r})$ be the weak solution to \eqref{eq1}.
Let $\kappa_1=\kappa_1(n,p,\Lambda)>0$ be the constant defined in \eqref{exist.M.cond}, with $\tilde{q}$ as in \eqref{tilde.q}. If
$$[\log\mathbb{M}]_{\mathrm{BMO}({B_{3r}})}<\kappa_1,$$
then
\begin{equation}\label{est.lem.4.4}
\fint_{B_{2r}} |\A{Du}-\A{Dv}|^{\tilde{q}}\,dx 
\le c\left\{\left(\frac{|\mu|(B_{2r})}{r^{n-1}}\right)^{\frac{\tilde{q}}{p-1}}+\chi_{\{p<2\}}\left(\frac{|\mu|(B_{2r})}{r^{n-1}}\right)^{\tilde{q}}\left(\fint_{B_{2r}} |\A{Du}|^{\tilde{q}}\,dx\right)^{2-p}\right\},
\end{equation}
where $c=c(n,p,\Lambda)>0$.
\end{lem}
\begin{proof}
The proof is similar to that of {\bf Step 1} in Section~3. Hence we only emphasize the necessary modifications. 

We set $\xi=\frac{n-\tilde{q}+n(p-\tilde{q})}{2(n-\tilde{q})}.$ Since $\tilde{q}$ is defined as \eqref{tilde.q}, \eqref{param.xi} and \eqref{param.xi.1} are satisfied. For any constant $h>0$, testing \eqref{main.eq} and \eqref{eq1} with $\varphi = h^{1-\xi} - (h+(u-v)_{\pm})^{1-\xi} \in W^{1,p}_{0,\omega}(B_{2r})$, we obtain
\begin{equation}\label{eq38}
\fint_{B_{2r}} \frac{\langle|\mathbb{M} D u|^{p-2} \mathbb{M} Du - |\mathbb{M} Dv|^{p-2} \mathbb{M} Dv , \mathbb{M} Du - \mathbb{M} Dv \rangle }{(h+|u-v|)^{\xi}}\,dx
\le c \frac{h^{1-\xi}}{\xi-1} \frac{|\mu|(B_{2r})}{r^n}.
\end{equation}

When $p \geq 2$, it follows from \eqref{mono.ineq2} and \eqref{eq38} that
\begin{equation*}
\fint_{B_{2r}}
\frac{\omega^p |Du-Dv|^p}{(h+|u-v|)^\xi}\,dx
\le
c \frac{h^{1-\xi}}{\xi-1} \frac{|\mu|(B_{2r})}{r^n}.
\end{equation*}
At this point, the remaining part in this case follows the same lines as in {\bf Step 1} of the proof of Theorem~\ref{Existence.of.SOLA}, with $u$ replaced by $u-v$ and $\Omega$ replaced by $B_{2r}$. Consequently, we obtain
$$
\fint_{B_{2r}} |\A{Du}-\A{Dv}|^{\tilde{q}}\,dx \le c\left(\frac{|\mu|(B_{2r})}{r^{n-1}}\right)^{\frac {\tilde{q}}{p-1}}.
$$

It remains to consider the case $2-{\frac1n}<p < 2$. In this case, we provide the proof in detail based on monotonicity estimate in \eqref{mono.ineq2}.
Since $\tilde q 
<2$, H\"older's inequality with the conjugate exponents $\frac{2}{\tilde q}$ and $\frac{2}{2-\tilde q}$, together with \eqref{mono.ineq2} and \eqref{eq38}, yields
\begin{equation}\label{eq40}
\begin{aligned}
\fint_{B_{2r}} |\A{Du}-\A{Dv}|^{\tilde{q}} \, dx
&\le c\left[\fint_{B_{2r}}\frac{\omega^{p}|Du-Dv|^2}{(h+|u-v|)^\xi}(|Du|+|Dv|)^{p-2}\, dx\right]^\frac{\tilde{q}}{2}\mathbb{I}\\
&\le c \left[\frac{h^{1-\xi}}{\xi-1} \frac{|\mu|(B_{2r})}{r^n}\right]^\frac{\tilde{q}}{2}
\mathbb{I},
\end{aligned}
\end{equation}
where we denote
$$
\mathbb{I}:=\left(\fint_{B_{2r}}\omega^{\frac{(3-p)\tilde q p'}{2-{\tilde q}}}
(|Du|+|Dv|)^{\frac{(2-p){\tilde q}}{2-{\tilde q}}}
(h+|u-v|)^{\frac{\xi \tilde{q}}{2-\tilde{q}}}
\, dx\right)^\frac{2-\tilde{q}}{2}.
$$
To estimate $\mathbb{I}$, we take
$$
h=\left(\fint_{B_{2r}}\omega^{\frac{\tilde{q}p'}{p-\tilde{q}}}|u-v|^{\frac{\xi \tilde{q}}{p-\tilde{q}}}\,dx\right)^{\frac{p-\tilde{q}}{\xi \tilde{q}}}
\left(\fint_{B_{2r}}\omega^{\frac{\tilde{q}p'}{p-\tilde{q}}}\,dx\right)^{\frac{\tilde{q}-p}{\xi \tilde{q}}}.$$
Since $\tilde{q}<p$, by applying H\"older's inequality with exponents $\frac{2-\tilde{q}}{2-p}$ and $\frac{2-\tilde{q}}{p-\tilde{q}}$, we obtain
\begin{equation}\label{eq41}
\begin{aligned}
\mathbb{I}
&\le \left(\fint_{B_{2r}} \omega^{\tilde{q}p'}\bigl(|Du|+|Dv|\bigr)^{\tilde{q}}\,dx
\right)^{\frac{2-p}{2}}\left(\fint_{B_{2r}}\omega^{\frac{\tilde{q} p'}{p-\tilde{q}}}(h+|u-v|)^{\frac{\xi \tilde{q}}{p-\tilde{q}}}\,dx\right)^{\frac{p-\tilde{q}}{2}}\\
&\le c  h^\frac{\xi \tilde{q}}{2}
\left(\fint_{B_{2r}}\omega^{\frac{\tilde{q} p'}{p-\tilde{q}}}\,dx\right)^{\frac{p-\tilde{q}}{2}} 
\left[\left(\fint_{B_{2r}}{|\A{Du}|}^{\tilde{q}}\,dx\right)^{\frac{2-p}{2}}+\left(\fint_{B_{2r}}|\A{Du}-\A{Dv}|^{\tilde{q}}\,dx\right)^{\frac{2-p}{2}}\right].
\end{aligned}
\end{equation}

It remains to estimate $h$. Repeating the argument used to derive \eqref{exist.bd.1}, we obtain
\begin{equation}\label{est.h.2}
h
\le c r \left( \fint_{B_{2r}} |\A{Du}-\A{Dv}|^{\tilde{q}} \dx \right)^{\frac{1}{\tilde{q}}}
\left( \fint_{B_{2r}} \omega^{- \frac{n\tilde{q}p'}{n-\tilde{q}}}\;dx \right)^{\frac{n(p-\tilde{q})-(n-\tilde{q})\xi}{n\tilde{q}\xi }}
\left(\fint_{B_{2r}}\omega^{\frac{\tilde{q}p'}{p-\tilde{q}}}\,dx\right)^{\frac{\tilde{q}-p}{\xi \tilde{q}}}.
\end{equation}

Combining \eqref{eq40}, \eqref{eq41}, and \eqref{est.h.2}, and then applying Young's inequality, we derive
\begin{equation*}\label{eq44}
\begin{aligned}
&\fint_{B_{2r}}|\A{Du}-\A{Dv}|^{\tilde{q}}\,dx\\
&\le c\left(\frac{|\mu|(B_{2r})}{r^{n-1}}\right)^{\tilde{q}}
\left(\fint_{B_{2r}}|\A{Du}|^{\tilde{q}}\,dx\right)^{2-p}
+c\left(\frac{|\mu|(B_{2r})}{r^{n-1}}\right)^{\tilde{q}}
\left(\fint_{B_{2r}}|\A{Du}-\A{Dv}|^{\tilde{q}}\,dx\right)^{2-p}\\
&\le
{\frac12} \fint_{B_{2r}}|\A{Du}-\A{Dv}|^{\tilde{q}}\,dx
+ c\left(\frac{|\mu|(B_{2r})}{r^{n-1}}\right)^\frac{\tilde{q}}{p-1}
+c\left(\frac{|\mu|(B_{2r})}{r^{n-1}}\right)^{\tilde{q}}
\left(\fint_{B_{2r}}|\A{Du}|^{\tilde{q}}\,dx\right)^{2-p},
\end{aligned}
\end{equation*}
where $c=c(n,p, \Lambda)>0$. Here in the first inequality, we also have to use the Lemma~\ref{lem:logBMO}.

By absorbing the corresponding term into the left-hand side, we obtain the desired estimate \eqref{est.lem.4.4} in case $2-{\frac1n}<p\le 2$.
This completes the proof of Lemma~\ref{comparison estimate}.
\end{proof}

Finally, combining Lemma~\ref{reverse.Du} and Lemma~\ref{comparison estimate}, we arrive at the following $L^1$-comparison estimate.
\begin{coro}\label{comparison.2}
Assume that \eqref{p} holds. Let $u\in W^{1,p}_{0,\omega}(\Omega)$ be the weak solution to \eqref{main.eq} and $v\in u+W^{1,p}_{0,\omega}(B_{2r})$ be the weak solution to \eqref{eq1}. 
Let ${\kappa_4}={\kappa_4}(n,p,\Lambda)>0$ be the constant defined in \eqref{gamma4}. If
$$[\log\mathbb{M}]_{\mathrm{BMO}({B_{3r}})}<{\kappa_4},$$
then
\begin{equation*}
\fint_{B_{2r}} |\A{Du}-\A{Dv}|\,dx 
\le c\left\{\left(\frac{|\mu|(B_{3r})}{r^{n-1}}\right)^{\frac{1}{p-1}}+\chi_{\{p<2\}}\left(\frac{|\mu|(B_{3r})}{r^{n-1}}\right)\left(\fint_{B_{3r}} |\A{Du}|\,dx\right)^{2-p}\right\},
\end{equation*}
where $c=c(n,p,\Lambda)>0.$
\end{coro}

\section{Proof of Theorem~\ref{theo:main}}
We fix a ball $B_{2R}\Subset\Omega$ with $2R\le R_0$.
Denote $\kappa=\kappa(n,p,\Lambda):=\frac{\kappa_4}{2}$ and assume
\begin{equation*}\label{final.gamma}
[\log\mathbb{M}]_{\operatorname{BMO}_{R_0}(\mathbb{R}^n)}
<{\frac{{\kappa}}{q}},
\end{equation*}
where $\kappa_4$ is defined in \eqref{gamma4}. Recall from \eqref{gamma2} and \eqref{gamma4} that ${\kappa_4} \le \min\left\{\kappa_1,\frac{\kappa_2}{p}\right\} \le  \kappa_3$
and
\begin{equation}\label{kappa.cond}
\frac{\kappa}{q}<\frac{\kappa_4}{q+1}\le\frac{\kappa_2}{p(q+1)}.
\end{equation}
Therefore, the reverse H\"older inequality and comparison estimates established in Section 4 remain valid for the approximating problems.



\subsection{Covering argument}

We begin by introducing the constant
\begin{equation}\label{511}
\lambda_0:=\fint_{B_{2R}} |\A{Du}|\,dx
+\frac{1}{\tau} \fint_{B_{2R}} \left(\mathcal{M}_1(\mu)\right)^{\frac{1}{p-1}}\,dx
\end{equation}
for some $\tau \in (0,1)$ chosen later.
Moreover, for $1\le t\le 2$ and $\lambda>0$, we denote the super-level set of $\A{Du}$ in $B_{tR}$ by
\begin{equation*}\label{512}
E(t,\lambda):=\bigl\{x\in B_{tR}: |\A{Du}|>\lambda\bigr\}.
\end{equation*}
We now fix $1\le t_1<t_2\le 2$.
A slight modification of the argument in {\bf Step 1} of \cite[Section 9]{GM2007} yields the following covering lemma.

\begin{lem}\label{covering}
Let 
\begin{equation*}\label{513}
\lambda>A\lambda_0,\; where \;\;A := \left(\frac{{60}}{t_2-t_1}\right)^{n}.
\end{equation*}
 Then there exists a countable family of pairwise disjoint balls
$\{B_{R_i}(x_i)\}_{i\ge1}$ with $x_i\in B_{t_1 R}$ and $0< R_i \le \frac{(t_2-t_1)R}{{30}}$ such that
$$
E(t_1,\lambda)\setminus \mathcal{N} \subset \bigcup_{i=1}^{\infty} B_{5R_i}(x_i)
\subset B_{t_2 R},
$$
where \(\mathcal{N}\subset \mathbb{R}^n\) is a Lebesgue measure zero set.
Moreover, for each \(i\in\mathbb{N}\),
\begin{equation}\label{514}
\fint_{B_{R_i}(x_i)} |\A{Du}|\,dx
+\frac{1}{\tau}\fint_{B_{R_i}(x_i)} \left(\mathcal{M}_1(\mu)\right)^\frac{1}{p-1}\,dx
=\lambda,
\end{equation}
and 
\begin{equation}\label{515}
\fint_{B_{\rho}(x_i)} |\A{Du}|\,dx
+\frac{1}{\tau}\fint_{B_{\rho}(x_i)} \left(\mathcal{M}_1(\mu)\right)^\frac{1}{p-1}\,dx
<\lambda,
\end{equation}
for every $\rho\in\left(R_i,(t_2-t_1)R\right]$.
\end{lem}

As a direct consequence of Lemma~\ref{covering}, we obtain the following estimate. Indeed, for each ball $B_{R_i}(x_i)$, \eqref{514} implies
\begin{equation}\label{516}
\fint_{B_{R_i}(x_i)} |\A{Du}|\,dx \ge \frac{\lambda}{2}
\quad\text{or}\quad
\frac{1}{\tau}\fint_{B_{R_i}(x_i)}
\left(\mathcal{M}_1(\mu)\right)^{\frac{1}{p-1}}\,dx
\ge \frac{\lambda}{2}.
\end{equation}
Moreover, since ${15}R_i\in\left(R_i,(t_2-t_1)R\right]$, taking $\rho={15}R_i$ in \eqref{515} yields
\begin{equation}\label{517}
\fint_{B_{{15}R_i}(x_i)} |\A{Du}|\,dx<\lambda,
\quad\text{and}\quad
\frac{1}{\tau}\fint_{B_{{15}R_i}(x_i)} \left(\mathcal{M}_1(\mu)\right)^{\frac{1}{p-1}}\,dx
<\lambda.
\end{equation}

Recalling that $u$ is a SOLA to \eqref{main.eq}, there exists an approximating sequence of weak solutions $\{u_k\}_{k\ge1}$ to \eqref{weak-1} with a sequence of data $\{\mu_k\}_{k\ge1}$. 
Then by \eqref{weak-2}, we have
\begin{equation*}
 \lim_{k\to\infty} \int_\Omega|\A{Du}-\A{Du_k}|\,dx=0.
 \end{equation*}
Consequently, for each ball $B_{{15}R_i}(x_i)$ defined in Lemma~\ref{covering} and any $\epsilon\in(0,1)$, there exists $k_i\ge1$ such that for all $k\ge k_i$,
\begin{equation}\label{518}
\fint_{B_{{15}R_i}(x_i)} |\A{Du}-\A{Du_k}|\,dx <\epsilon\lambda_0.
\end{equation}

\subsection{Estimates on super-level sets}
For the constant $T>1$ chosen shortly, we take the pairwise disjoint family of balls  $\{B_{R_i}(x_i)\}_{i\ge 1}$ obtained by Lemma~\ref{covering} with $E(t_1,\lambda)$.
Then up to a Lebesgue measure zero set, we have
$$
E(t_1,T\lambda)\subset E(t_1,\lambda)\subset \bigcup_{i=1}^\infty B_{5R_i}(x_i)\subset B_{{t_2}R}.
$$
Therefore, for each $i\in\mathbb{N}$, by denoting
\begin{equation*}\label{def.E}
E_{5R_i}(x_i):=E(t_1,T\lambda)\cap B_{5R_i}(x_i)=\bigl\{x\in B_{t_1 R}\cap B_{5R_i}(x_i):\ |\A{Du}|>T\lambda\bigr\},
\end{equation*}
we get
\begin{equation}\label{522}
\int_{E(t_1,T\lambda)} |\A{Du}|\,dx
\le \sum_{i=1}^\infty \int_{E_{5R_i}(x_i)} |\A{Du}|\,dx.
\end{equation}

Now we fix $i \in \N$. For each $k\ge k_i$, let $v_k \in u_k+W^{1,p}_{0,\omega}(B_{{10}R_i})$ be the weak solution to the homogeneous problem \eqref{eq1}.

We combine \eqref{weak-3}, \eqref{517}, and \eqref{518} to get
\begin{equation}\label{523}
\fint_{B_{{15}R_i}(x_i)} |\A{Du_k}|\,dx<c\lambda
\quad\text{and}\quad
\fint_{B_{{15}R_i}(x_i)} \left(\mathcal{M}_1(\mu_k)\right)^{\frac{1}{p-1}}\,dx
<c\tau\lambda.
\end{equation}
Moreover, for a.e. $x\in E_{5R_i}(x_i)$, we have $|\A{Du}|>T\lambda$. 
Hence
\begin{equation*}
\begin{aligned}
|\A{Du}|
&\le |\A{Du}-\A{Du_k}|+|\A{Du_k}-\A{Dv_k}|+|\A{Dv_k}|\\
&\le |\A{Du}-\A{Du_k}|+|\A{Du_k}-\A{Dv_k}|+|\A{Dv_k}|\left(\frac{|\A{Du}|}{T\lambda}\right)^{\frac{1}{(q+1)'}}\\
&\le|\A{Du}-\A{Du_k}|+|\A{Du_k}-\A{Dv_k}|+\frac{c}{(T\lambda)^{q}}|\A{Dv_k}|^{q+1}+\frac12|\A{Du}|.
\end{aligned}
\end{equation*}
Absorbing the last term into the left-hand side yields
$$|\A{Du}|\le c|\A{Du}-\A{Du_k}|+c|\A{Du_k}-\A{Dv_k}|+\frac{c}{(T\lambda)^{q}}|\A{Dv_k}|^{q+1}.$$
Consequently, Lemma~\ref{Reverse.H}, Corollary~\ref{comparison.2}, \eqref{kappa.cond}, \eqref{518}, and \eqref{523} imply that
\begin{equation}\label{525}
\begin{aligned}
&\int_{E_{5R_i}(x_i)} |\A{Du}|\,dx \\
&\le c|B_{R_i}(x_i)|\fint_{B_{5R_i}(x_i)} |\A{Du}-\A{Du_k}|\,dx
+c|B_{R_i}(x_i)|\fint_{B_{5R_i}(x_i)} |\A{Du_k}-\A{Dv_k}|\,dx\\
&\quad+\frac{c|B_{R_i}(x_i)|}{(T\lambda)^{q}}\left(\fint_{B_{10R_i}(x_i)}|\A{Dv_k}|\,dx\right)^{q+1}\\
&\le c|B_{R_i}(x_i)|\fint_{B_{5R_i}(x_i)} |\A{Du}-\A{Du_k}|\,dx
+c|B_{R_i}(x_i)|\fint_{B_{5R_i}(x_i)} |\A{Du_k}-\A{Dv_k}|\,dx\\
&\quad+\frac{c|B_{R_i}(x_i)|}{(T\lambda)^{q}}\left(\fint_{B_{10R_i}(x_i)}|\A{Du_k}-\A{Dv_k}|\,dx\right)^{q+1}
+\frac{c|B_{R_i}(x_i)|}{(T\lambda)^{q}}\left(\fint_{B_{10R_i}(x_i)}|\A{Du_k}|\,dx\right)^{q+1}\\
&\le c\left(\epsilon+\tau+\tau^{p-1}+\frac{\tau^{q+1}+\tau^{(p-1)(q+1)}+1}{T^{q}}\right)\lambda |B_{R_i}(x_i)|.
\end{aligned}
\end{equation}
We next use \eqref{516} to estimate $|B_{R_i}(x_i)|$. The first inequality therein gives
\begin{equation*}
\begin{aligned}
|B_{R_i}(x_i)|
&\le \frac{2}{\lambda}
\int_{B_{R_i}(x_i)\cap\{|\A{Du}|>\frac{\lambda}{4}\}} |\A{Du}|\,dx
+
\frac{2}{\lambda}
\int_{B_{R_i}(x_i)\cap\{|\A{Du}|\le\frac{\lambda}{4}\}} |\A{Du}|\,dx\\
&\le  \frac{2}{\lambda}
\int_{B_{R_i}(x_i)\cap\{|\A{Du}|>\frac{\lambda}{4}\}} |\A{Du}|\,dx
+\frac12\,|B_{R_i}(x_i)|.
\end{aligned}
\end{equation*}
It follows that
\begin{align}\label{526}
|B_{R_i}(x_i)|
\le \frac{4}{\lambda}
\int_{B_{R_i}(x_i)\cap\{|\A{Du}|>\frac{\lambda}{4}\}} |\A{Du}|\,dx.
\end{align}
Similarly, the second inequality in \eqref{516} yields
\begin{equation}\label{527}
\begin{aligned}
|B_{R_i}(x_i)|
&\le 
\frac{4}{\lambda}\int_{B_{R_i}(x_i)\cap\{\frac{\left(\mathcal{M}_1(\mu)\right)^{1/(p-1)}}{\tau}>\frac{\lambda}{4}\}} \frac{\left(\mathcal{M}_1(\mu)\right)^\frac{1}{p-1}}{\tau}\,dx.
\end{aligned}
\end{equation}
Combining \eqref{526} and \eqref{527}, we obtain
\begin{equation*}
|B_{R_i}(x_i)|
\le \frac{4}{\lambda}
\int_{B_{R_i}(x_i)\cap\{|\A{Du}|>\frac{\lambda}{4}\}} |\A{Du}|\,dx
+\frac{4}{\lambda}\int_{B_{R_i}(x_i)\cap\{\frac{\left(\mathcal{M}_1(\mu)\right)^{1/(p-1)}}{\tau}>\frac{\lambda}{4}\}} \frac{\left(\mathcal{M}_1(\mu)\right)^\frac{1}{p-1}}{\tau}\,dx.
\end{equation*}
Applying \eqref{525} and the preceding estimate to \eqref{522}, we arrive at
\begin{equation}\label{5210}
\begin{aligned}
&\int_{E(t_1,T\lambda)} |\A{Du}|\,dx\\
&\le c\left(\epsilon+\tau+\tau^{p-1}+\frac{\tau^{q+1}+\tau^{(p-1)(q+1)}+1}{T^{q}}\right) 
\int_{B_{t_2 R}\cap\{|\A{Du}|>\frac{\lambda}{4}\}} |\A{Du}|\,dx\\
&\quad+ c\left(\epsilon+\tau+\tau^{p-1}+\frac{\tau^{q+1}+\tau^{(p-1)(q+1)}+1}{T^{q}}\right) 
\int_{B_{t_2 R}\cap\{\frac{\left(\mathcal{M}_1(\mu)\right)^{1/(p-1)}}{\tau}>\frac{\lambda}{4}\}} \frac{\left(\mathcal{M}_1(\mu)\right)^\frac{1}{p-1}}{\tau}\,dx,
\end{aligned}
\end{equation}
where $c=c(n,p,q,\Lambda)>0$.

\subsection{Proof of  the estimate \eqref{main.thm.est}}
For $q>1$ and $h>AT\lambda_0$, we define the truncated function of $|\A{Du}|$ by
$$
|\A{Du}|_h := \min\bigl\{|\A{Du}|,\; h\bigr\}.
$$
By a standard manipulation, we obtain
\begin{equation}\label{531}
\begin{aligned}
&\int_{B_{t_1 R}} |\A{Du}|_{h}^{\,q-1}\,|\A{Du}|\,dx\\ 
&=(q-1)T^{q-1}\int_{B_{t_1 R}}\int_{0}^{\frac{h}{T}}
\lambda^{q-2}\,\chi_{\{|\A{Du}|>T\lambda\}}\,d\lambda\;|\A{Du}|\,dx \\
&=(q-1)T^{q-1}\int_{0}^{\frac{h}{T}}
\lambda^{q-2}\int_{E(t_1,T\lambda)}|\A{Du}|\,dx\;d\lambda\\
&=(q-1)T^{q-1}\int_{0}^{A\lambda_0}
\lambda^{q-2}\int_{E(t_1,T\lambda)}|\A{Du}|\,dx\;d\lambda
+(q-1)T^{q-1}\int_{A\lambda_0}^{\frac{h}{T}}
\lambda^{q-2}\int_{E(t_1,T\lambda)}|\A{Du}|\,dx\;d\lambda\\
&=:\mathbb{I}_1+\mathbb{I}_2.
\end{aligned}
\end{equation}

$\mathbb {I}_1$ is estimated by
\begin{equation}\label{533}
\begin{aligned}
\mathbb{I}_1
\le (q-1)T^{q-1}\int_{0}^{A\lambda_0}\lambda^{q-2}d\lambda
\int_{B_{t_1R}}|\A{Du}|\,dx
\le \frac{c{\lambda_0}^{q-1}}{(t_2-t_1)^{n(q-1)}}
\int_{B_{2R}} |\A{Du}|\,dx.
\end{aligned}
\end{equation}

We next estimate $\mathbb {I}_2$ by means of \eqref{5210} as follows:
\begin{equation*}
\begin{aligned}
\mathbb{I}_2
&\le
c\left[\left(\epsilon+\tau+\tau^{p-1}\right)T^{q-1}+\frac{\tau^{q+1}+\tau^{(p-1)(q+1)}+1}{T}\right] \\
& \quad \times \left[ \int_{0}^{4h} \lambda^{q-2}
\int_{B_{t_2 R}\cap\{|\A{Du}|>\frac{\lambda}{4}\}} |\A{Du}|\,dx\;d\lambda \right. \\
&\qquad+ \left. \int_{0}^{\infty} \lambda^{q-2}
\int_{B_{t_2 R}\cap\{\frac{\left(\mathcal{M}_1(\mu)\right)^{1/(p-1)}}{\tau}>\frac{\lambda}{4}\}} \frac{\left(\mathcal{M}_1(\mu)\right)^\frac{1}{p-1}}{\tau}\,dx\;d\lambda \right].
\end{aligned}
\end{equation*}
{Noting that $c=c(n, p, q, \Lambda)>0$, we can take $T$ large enough and $\epsilon, \tau$ small enough to satisfy}
\begin{equation}\label{535}
\begin{aligned}
\mathbb{I}_2
\le {\frac12}
\int_{B_{t_2 R}} |\A{Du}|_{h}^{q-1}\,|\A{Du}|\,dx
+c\int_{B_{2 R}} \left(\mathcal{M}_1(\mu)\right)^\frac{q}{p-1}\,dx.
\end{aligned}
\end{equation}

Substituting \eqref{533} and \eqref{535} into \eqref{531},  we obtain
\begin{equation*}\label{536}
\begin{aligned}
&\int_{B_{t_1 R}} |\A{Du}|_{h}^{\,q-1}\,|\A{Du}|\,dx \\
&\le\frac{c{\lambda_0}^{q-1}}{(t_2-t_1)^{n(q-1)}}
\int_{B_{2R}} |\A{Du}|\,dx
+{\frac12}
\int_{B_{t_2 R}} |\A{Du}|_{h}^{q-1}\,|\A{Du}|\,dx
+c
\int_{B_{2 R}} \left(\mathcal{M}_1(\mu)\right)^\frac{q}{p-1}\,dx.
\end{aligned}
\end{equation*}
Since $1\le t_1 <t_2\le2$, Lemma~\ref{lem-iteration} gives
\begin{equation}\label{539}
\begin{aligned}
\int_{B_{R}} |\A{Du}|_{h}^{\,q-1}\,|\A{Du}|\,dx 
\le 
c{\lambda_0}^{q-1}
\int_{B_{2 R}} |\A{Du}|\,dx
+c\int_{B_{2 R}} \left(\mathcal{M}_1(\mu)\right)^\frac{q}{p-1}\,dx.
\end{aligned}
\end{equation}
Therefore, by letting $h\to\infty$ in \eqref{539} and applying the monotone convergence theorem, we obtain
\begin{equation*}
\begin{aligned}
\fint_{B_{R}} |\A{Du}|^q\,dx 
\le 
c{\lambda_0}^{q-1}
\fint_{B_{2 R}} |\A{Du}|\,dx
+c\fint_{B_{2 R}} \left(\mathcal{M}_1(\mu)\right)^\frac{q}{p-1}\,dx.
\end{aligned}
\end{equation*}
Recalling the definition of $\lambda_0$ in \eqref{511}, we infer that
{\begin{equation*}
\begin{aligned}
\fint_{B_{R}} \,|\A{Du}|^{q}\,dx 
&\le c \left(\fint_{B_{2 R}} |\A{Du}|\,dx\right)^q
+c\fint_{B_{2 R}} \left(\mathcal{M}_1(\mu)\right)^\frac{q}{p-1}\,dx\\
&\quad +c\left(\fint_{B_{2 R}} |\A{Du}|\,dx\right)
\left( \fint_{B_{2R}} \left(\mathcal{M}_1(\mu)\right)^\frac{1}{p-1}\,dx\right)^{q-1} \\
&\le c \left(\fint_{B_{2 R}} |\A{Du}|\,dx\right)^q
+c\fint_{B_{2R}} \left(\mathcal{M}_1(\mu)\right)^\frac{q}{p-1}\,dx.
\end{aligned}
\end{equation*}}
This completes the proof of \eqref{main.thm.est}.

\,

\noindent\textbf{Disclosure statement}

\noindent The authors report there are no competing interests to declare.

\,

\noindent\textbf{Funding.}

\noindent Sun-Sig Byun was supported by Mid-Career Bridging Program through Seoul National University.
Xiatong Li was supported by National Research Foundation of Korea (NRF) grant funded by the Korean
government [Grant No. 2022R1A2C1009312].
Yeonghun Youn was supported by National Research Foundation of Korea (NRF) grant funded by the Korean government [Grant No. RS-2025-00555316]. 

\,

\noindent\textbf{Data availability statement.}

\noindent This paper has no associated data.

\end{document}